\documentclass[11pt,reqno]{amsart}
\usepackage{amsmath, amsthm, amssymb, amsbsy}
\usepackage[letterpaper=true,colorlinks=true,urlcolor=blue,linkcolor=blue,citecolor=blue]{hyperref}

\newtheorem{lemma}{Lemma}[section]
\newtheorem{theorem}{Theorem}[section]
\newtheorem{proposition}{Proposition}[section]
\newtheorem{corollary}{Corollary}[section]
\newtheorem*{remark}{Remark}

\theoremstyle{definition}

\numberwithin{equation}{section}

\def\le{\leqslant} \def\ge{\geqslant}

\def\pmod #1{({\rm mod}\ #1)}

\def \N {{\mathbb N}}
\def \P {{\mathbb P}}
\def \Q {{\mathbb Q}}
\def \R {{\mathbb R}}

\def \Z {{\mathbb Z}}

\newcommand{\rank}{{\rm{rank}}}
\newcommand{\rankoff}{{\rm{rank}_{\rm{off}}}}
\newcommand{\dia}{{\rm{diag}}}

\begin{document}
\title[]{The quadratic form in 9 prime variables}
\author{Lilu  Zhao}
\email{zhaolilu@gmail.com}
\address{School of Mathematics, Hefei University of Technology, Heifei 230009, People's Republic of China}

\begin{abstract}
Let $f(x_1,\ldots,x_n)$ be a regular indefinite integral quadratic
form with $n\ge 9$, and let $t$ be an integer. It is established
that $f(x_1,\ldots,x_n)=t$ has solutions in prime variables if
there are no local obstructions.
\end{abstract}

\maketitle


{\let\thefootnote\relax}

{\let\thefootnote\relax}

\section{Introduction}

Let $A=(a_{i,j})_{1\le i,j\le n}$ a symmetric integral matrix with
$n\ge 4$. In other words,
\begin{align}\label{defA} A=\begin{pmatrix}a_{1,1}  & \cdots & a_{1,n}
\\ \vdots & \cdots & \vdots
\\ a_{n,1} & \cdots  & a_{n,n} \end{pmatrix}\end{align}
with $a_{i,j}=a_{j,i}\in\Z$ for all $1\le i< j\le n$. Let
$f(x_1,\ldots,x_n)$ be the quadratic form defined as
\begin{align}\label{fn}f(x_1,\ldots,x_n)=\sum_{i=1}^{n}\sum_{j=1}^na_{i,j}x_ix_j.\end{align}
Let $t$ be an integer. For indefinite integral quadratic forms,
the Hasse principle asserts that $f(x_1,\ldots,x_n)=t$ has integer
solutions if and only if $f(x_1,\ldots,x_n)=t$ has local
solutions.

In this paper, we consider the equation $f(x_1,\ldots,x_n)=t$,
where $x_1,\ldots,x_n$ are prime variables. It is expected that
$f(x_1,\ldots,x_n)=t$ has solutions with $x_1,\ldots,x_n$ primes
if there are suitable local solutions. A classical theorem of Hua
\cite{Hua} deals with diagonal quadratic forms in 5 prime
variables. In particular, all sufficiently large integers,
congruent to 5 modulo 24, can be represented as a sum of five
squares of primes. Recently, Liu \cite{Liu} handled a wide class
of quadratic forms $f$ with $10$ or more prime variables. The
general quadratic form in prime variables (or in dense sets) was
recently investigated by Cook \cite{Cook}, and by Keil
\cite{Keil}. In particular, the work of Keil \cite{Keil} handled
all regular quadratic forms in $17$ or more variables. It involves
only five primes variables for diagonal quadratic equation due to
the effective mean value theorem. This is similar to the problem
concerning diophantine equations for cubic forms. The works of
Baker \cite{Baker}, Vaughan \cite{V1,V2} and Wooley \cite{W1,W2}
can deal with the diagonal cubic equation with 7 variables.
However, more variables are involved for general cubic forms. One
can refer to the works of Heath-Brown \cite{HB1,HB} and Hooley
\cite{Hooley} for general cubic forms.

The purpose of this paper is to investigate general regular
quadratic forms in $9$ or more prime variables. We define
\begin{align*}N_{f,t}(X)=\sum_{\substack{1\le x_1,\ldots,x_n\le X \\ f(x_1,\ldots,x_n)=t}}\prod_{j=1}^n
\Lambda(x_j),\end{align*} where $\Lambda(\cdot)$ is the von
Mangoldt function. Our main result is the following.
\begin{theorem}\label{theorem1}
Let $f(x_1,\ldots,x_n)$ be a quadratic form given by (\ref{fn}),
and let $t\in \Z$. Let $\mathfrak{S}(f,t)$ and
$\mathfrak{I}_{f,t}(X)$ be defined in (\ref{defSeries}) and
(\ref{defIntegral}), respectively. Suppose that $\rank(A)\geqslant
9$, and $K$ is an arbitrary large real number. Then we have
\begin{align}\label{asymptotic}N_{f,t}(X)=\mathfrak{S}(f,t)\mathfrak{I}_{f,t}(X)+O(X^{n-2}\log^{-K}X),\end{align}
where the implied constant depends on $f$ and $K$.
\end{theorem}

We write $V_{f,t}$ for the affine quadric
$$\{(x_1,\ldots,x_n)^{T}\in \Z^n:\ f(x_1,\ldots,x_n)=t\}.$$ For a
set $S$, we define
\begin{align*}V_{f,t}(S)=\{(x_1,\ldots,x_n)^{T}\in S^n:\ f(x_1,\ldots,x_n)=t\}.\end{align*}
Denote by $\P$ the set of all prime numbers. For a prime $p\in
\P$, we use $\Z_p$ to denote the ring of $p$-adic integers. Let
\begin{align*}V_{f,t}^0(\Z_p)=\{(x_1,\ldots,x_n)^{T}\in V_{f,t}(\Z_p):\ x_1\cdots x_n\equiv 0\pmod{p}\}.\end{align*}
We say there are no local obstructions for $V$ if for any $p\in
\P$ one has
\begin{align*}V_{f,t}(\Z_p)=V_{f,t}^0(\Z_p).\end{align*}
The general local to global conjecture of Bourgain-Gamburd-Sarnak
\cite{BGS} asserts that $V_{f,t}(\P)$ is Zariski dense in
$V_{f,t}$ if and only if there are no local obstructions for
$V_{f,t}$. Theorem 1.1 of Liu \cite{Liu} verified this conjecture
for a wide class of regular indefinite integral quadratic forms
with 10 or more variables. Theorem \ref{theorem1} has the
following corollary, which improves upon Theorem 1.1 of Liu
\cite{Liu}.
\begin{corollary}\label{cor1}
Let $f(x_1,\ldots,x_n)$ be a regular indefinite integral quadratic
form with $n\geqslant 9$, and let $t\in \Z$. Then $V_{f,t}(\P)$ is
Zariski dense in $V_{f,t}$ if and only if there are no local
obstructions.
\end{corollary}  Corollary \ref{cor1} covers all regular indefinite integral quadratic
forms in 9 prime variables. The $O$-constant in the asymptotic
formula (\ref{asymptotic}) is independent of $t$. Therefore,
Theorem \ref{theorem1} implies the following result.
\begin{corollary}\label{cor2}
Let $f(x_1,\ldots,x_n)$ be a positive definite integral quadratic
form with $n\geqslant 9$. Then there exist $r, q\in \N$ so that
all sufficiently large natural numbers $N$, congruent to $r$
modulo $q$, can be represented as $N=f(p_1,\ldots,p_n)$, where
$p_1,\ldots,p_n$ are prime numbers.
\end{corollary}

The method in this paper can be also applied to refine Theorem 1.1
of Keil \cite{Keil}. In particular, one can obtain a variant of
Theorem 1.1 of Keil \cite{Keil} for a wide class of quadratic
forms in 9 variables.

 \vskip3mm

\section{Notations}

As usual, we write $e(z)$ for $e^{2\pi iz}$. Throughout we assume
that $X$ is sufficiently large. Let $L=\log X$. We use $\ll$ and
$\gg$ to denote Vinogradov's well-known notation, while the
implied constants may depend on the form $f$. Denote by $\phi(q)$
the Euler function.

For a set $\mathcal{S}$, we denote by
\begin{align}\mathcal{S}^n=\{(x_1, \ldots,x_n)^{T}:\
x_1,\ldots,x_n\in \mathcal{S}\}.\end{align} We introduce the
notations for the set of $m$ by $n$ matrices
\begin{align}M_{m,n}(\mathcal{S})=\Big\{(a_{i,j})_{1\le i\le m,\,1\le j\le n}:\
a_{i,j}\in \mathcal{S}\Big\}\end{align} and the set of invertible
matrices of order $n$
\begin{align}GL_n(\mathcal{S})=\Big\{B\in M_{n,n}(\mathcal{S}):\
B \textrm{ is invertible}\Big\},\end{align} respectively. We
define the off-diagonal rank of $A$ as
\begin{align}\label{defoff}\rankoff(A)
=\max\{r:\ r\in R\},\end{align}where
\begin{align*}R=\Big\{\rank(B):\ B=(a_{i_k,j_l})_{1\le k,l\le r}\
\textrm{ with }\
\{i_1,\ldots,i_r\}\cap\{j_1,\cdots,j_r\}=\emptyset.\Big\}.\end{align*}
In other words, $\rankoff(A)$ is the maximal rank of a submatrix
in $A$, which does not contain any diagonal entries. For
$\mathbf{x}=(x_1,\ldots,x_n)^{T}\in \N^n$, we write
\begin{align*}\Lambda(\mathbf{x})=\Lambda(x_1)\cdots\Lambda(x_n).\end{align*}
For $\mathbf{x}=(x_1,\ldots,x_n)^{T}\in \Z^n$, we also use the
notation $\mathcal{A}(\mathbf{x})$ to indicate that the argument
$\mathcal{A}(x_j)$ holds for all $1\le j\le s$. The meaning will
be clear from the text. For example, we use $1\le \mathbf{x}\le X$
and $|\mathbf{x}|\le X$ to denote $1\le x_j\le X$ for $1\le j\le
n$ and $|x_j|\le X$ for $1\le j\le n$, respectively.

In order to apply the circle method, we introduce the exponential
sum
\begin{align}\label{defSalpha}S(\alpha)=
\sum_{1\le \mathbf{x}\le
X}\Lambda(\mathbf{x})e\big(\alpha\mathbf{x}{^T}A\mathbf{x}\big),\end{align}
where $A$ is defined in (\ref{defA}). We define
\begin{align}\label{defMQ}\mathcal{M}(Q)=\bigcup_{1\le q\le
Q}\bigcup_{\substack{1\le a\le q
 \\ (a,q)=1}}\mathcal{M}(q,a;Q),\end{align}
where
\begin{align*}\mathcal{M}(q,a;Q)=\Big\{\alpha:\ \big|\alpha-\frac{a}{q}\big|\le \frac{Q}{qX^2}\Big\}.\end{align*}
The intervals $\mathcal{M}(q,a;Q)$ are pairwise disjoint for $1\le
a\le q\le Q$ and $(a,q)=1$ provided that $Q\le X/2$. For $Q\le
X/2$, we set
\begin{align}\label{defmQ}\mathfrak{m}(Q)=\mathcal{M}(2Q)\setminus\mathcal{M}(Q).\end{align}
Now we introduce the major arcs defined as
\begin{align}\label{defM}\mathfrak{M}=\mathcal{M}(P) \ \ \textrm{ with }\ \ P=L^{K},\end{align}
where $K$ is a sufficiently large constant throughout this paper.
Then we define the minor arcs as
\begin{align}\label{defm}\mathfrak{m}=[X^{-1},1+X^{-1}]\setminus\mathfrak{M}.\end{align}

\vskip3mm
\section{The contribution from the major arcs}

For $q\in \N$ and $(a,q)=1$, we define \begin{align}
\label{defC}C(q,a)=\sum_{\substack{\mathbf{h}\in \N^n \\ 1\le
\mathbf{h}\le q
\\ (\mathbf{h},q)=1}}e\Big(\mathbf{h}^{T}A\mathbf{h}\frac{a}{q}\Big),\end{align}
where $A$ is given by (\ref{defA}). Let
\begin{align}
\label{defBq}B(q)=\frac{1}{\phi^n(q)}\sum_{\substack{1\le a\le q
\\ (a,q)=1}}C(q,a)e\Big(-\frac{at}{q}\Big).\end{align}
Concerning $B(q)$, we have the following two conclusions.
\begin{lemma}\label{lemma31}The arithmetic function $B(q)$ is multiplicative.
\end{lemma}
\begin{lemma}\label{lemma32}Let $B(q)$ be defined as (\ref{defBq}). If $\rank(A)\ge 5$, then we have
\begin{align*}
B(q)\ll_{A,\varepsilon}
q^{-3/2+\varepsilon}.\end{align*}\end{lemma}
Now we introduce the singular series $\mathfrak{S}(f,t)$ defined
as\begin{align}\label{defSeries}\mathfrak{S}(f,t)=\sum_{q=1}^\infty
B(q),\end{align} where $B(q)$ is given by (\ref{defBq}). From
Lemmas \ref{lemma31} and \ref{lemma32}, we conclude the following
result.
\begin{lemma}\label{lemma33}Suppose that $\rank(A)\ge 5$.
Then the singular series $\mathfrak{S}(f,t)$ is absolutely
convergent, and
\begin{align*}\mathfrak{S}(f,t)=\prod_{p\in \P}\chi_p(f,t),\end{align*}
where the local densities $\chi_p(f,t)$ are defined as
\begin{align*}\chi_p(f,t)=1+\sum_{m=1}^\infty
B(p^m).\end{align*} Moreover, if there are no local obstructions,
then one has\begin{align*}\mathfrak{S}(f,t)\gg 1.\end{align*}
\end{lemma}
We define
\begin{align}I(\beta)=\int_{[1,X]^n}e\big(\beta\mathbf{x}^{T}A\mathbf{x}\big)d\mathbf{x}.\end{align}
Since $I(\beta)\ll X^n(1+X^2|\beta|)^{-2}$ for $\rank(A)\ge 4$, we
introduce the singular integral
\begin{align}\label{defIntegral}&\mathfrak{I}_{f,t}(X)=\int_{-\infty}^{\infty}
I(\beta) e(-t\beta)d\beta,\end{align} where
$f(\mathbf{x})=\mathbf{x}^TA\mathbf{x}$.
\begin{lemma}\label{lemma34}Let $t\in \Z$, and let \begin{align*}S(\alpha)=
\sum_{1\le \mathbf{x}\le
X}\Lambda(\mathbf{x})e\big(\alpha\mathbf{x}{^T}A\mathbf{x}\big),\end{align*}
where $A\in M_{n,n}(\Z)$ is a symmetric matrix with $\rank(A)\ge
5$. Then one
has\begin{align}\label{major}\int_{\mathfrak{M}}S(\alpha)e(-t\alpha)d\alpha
=&\mathfrak{S}(f,t)\mathfrak{I}_{f,t}(X)+O(X^{n-2}L^{-K/4}).\end{align}\end{lemma}
\begin{proof}
We write $f(\mathbf{x})$ for $\mathbf{x}{^T}A\mathbf{x}$. By the
definition of $\mathfrak{M}$, one has
\begin{align}\label{major1}&\int_{\mathfrak{M}}S(\alpha)e(-t\alpha)d\alpha\notag
\\ =&\sum_{q\le P}\sum_{\substack{1\le a\le q\\ (a,q)=1}}\int_{|\beta|\le \frac{P}{qX^2}}
\sum_{1\le \mathbf{x}\le
X}\Lambda(\mathbf{x})e\Big(f(\mathbf{x})(\frac{a}{q}+\beta)\Big)e\Big(-t(\frac{a}{q}+\beta)\Big)d\beta.\end{align}
We introduce the congruence condition to deduce that
\begin{align*}&
\sum_{1\le \mathbf{x}\le
X}\Lambda(\mathbf{x})e\Big(f(\mathbf{x})(\frac{a}{q}+\beta)\Big)
\\= &\sum_{1\le \mathbf{h}\le q}e\Big(f(\mathbf{h})\frac{a}{q}\Big)
\sum_{\substack{1\le \mathbf{x}\le X\\
\mathbf{x}\equiv
\mathbf{h}\pmod{q}}}\Lambda(\mathbf{x})e\big(f(\mathbf{x})\beta\big)
\\= &\sum_{\substack{1\le \mathbf{h}\le q\\ (\mathbf{h},q)=1}}e\Big(f(\mathbf{h})\frac{a}{q}\Big)
\sum_{\substack{1\le \mathbf{x}\le X\\
\mathbf{x}\equiv
\mathbf{h}\pmod{q}}}\Lambda(\mathbf{x})e\big(f(\mathbf{x})\beta\big)+O(X^{n-1}L^2P).\end{align*}
Since $q\le P=L^{K}$, the Siegel-Walfisz theorem together with
integration by parts will imply for $(\mathbf{h},q)=1$ that
\begin{align*}
\sum_{\substack{1\le \mathbf{x}\le X\\
\mathbf{x}\equiv
\mathbf{h}\pmod{q}}}\Lambda(\mathbf{x})e\big(f(\mathbf{x})\beta\big)=&
\frac{1}{\phi^n(q)}\int_{[1,X]^n}e\big(f(\mathbf{x})\beta\big)d\mathbf{x}+O(X^{n}L^{-100K})
\\ =&
\frac{1}{\phi^n(q)}I(\beta)+O(X^{n}L^{-100K}).\end{align*} It
follows from above
\begin{align}\label{intand}
\sum_{1\le \mathbf{x}\le
X}\Lambda(\mathbf{x})e\Big(f(\mathbf{x})(\frac{a}{q}+\beta)\Big)
=&\frac{C(q,a)}{\phi^n(q)}I(\beta)+O(X^{n}L^{-10K}).\end{align} By
putting (\ref{intand}) into (\ref{major1}), we obtain
\begin{align}\label{major2}\int_{\mathfrak{M}}S(\alpha)e(-t\alpha)d\alpha
 =\sum_{q\le P}B(q)\int_{|\beta|\le \frac{P}{qX^2}}I(\beta)
 e(-t\beta)d\beta+O(X^{n}L^{-K}).\end{align} It follows from $I(\beta)\ll X^n(1+X^2|\beta|)^{-2}$ that
\begin{align}\label{upper}&\mathfrak{I}_{f,t}(X)\ll X^{n-2}\end{align}
and
\begin{align}\label{major3}&\int_{|\beta|\le \frac{P}{qX^2}}
I(\beta)e(-t\beta)d\beta=\mathfrak{I}_{f,t}(X)+O(qX^{n-2}P^{-1}).\end{align}
Combining (\ref{major2}), (\ref{upper}), (\ref{major3}) together
with Lemma \ref{lemma32}, we
conclude\begin{align*}\int_{\mathfrak{M}}S(\alpha)e(-t\alpha)d\alpha
=&\mathfrak{S}(f,t)\mathfrak{I}_{f,t}(X)+O(X^{n-2}L^{-K/4}).\end{align*}
The proof of Lemma \ref{lemma34} is completed.
\end{proof}

 \vskip3mm

\section{Estimates for exponential sums}
\begin{lemma}\label{lemma41}
Let $\{\xi_{z}\}$ be a sequence satisfying $|\xi_{z}|\le 1$. Then
one has
\begin{align*}\sum_{|y|\ll X}\Big|\sum_{|z|\ll X}\xi_{z}e(\alpha y z)\Big|^2\ll X\sum_{|x|\ll X}
\min\{X,\ \|x\alpha\|^{-1}\}.\end{align*}
\end{lemma}
\begin{proof}We expand the square to deduce that
\begin{align*}\sum_{|y|\ll X}\Big|\sum_{|z|\ll X}\xi_{z}e(\alpha y
z)\Big|^2=& \sum_{|z_1|\ll X}\sum_{|z_2|\ll
X}\xi_{z_1}\overline{\xi_{z_2}}\sum_{|y|\ll X}e\big(\alpha y
(z_1-z_2)\big)
\\ \le & \sum_{|z_1|\ll X}\sum_{|z_2|\ll
X}\Big|\sum_{|y|\ll X}e\big(\alpha y
(z_1-z_2)\big)\Big|.\end{align*} By changing variables, one can
obtain
\begin{align*}\sum_{|y|\ll X}\Big|\sum_{|z|\ll X}\xi_{z}e(\alpha y
z)\Big|^2 \ll & \sum_{|z|\ll X}\sum_{|x|\ll X}\Big|\sum_{|y|\ll
X}e\big(\alpha y x\big)\Big| \\ \ll & X\sum_{|x|\ll
X}\Big|\sum_{|y|\ll X}e\big(\alpha y x\big)\Big| \\ \ll &
X\sum_{|x|\ll X}\min\{X,\ \|x\alpha\|^{-1}\}.\end{align*} We
complete the proof.\end{proof}

\begin{lemma}\label{lemma42}For $\alpha\in \mathfrak{m}(Q)$, one has
\begin{align*}\sum_{|x|\ll X}\min\{X,\ \|x\alpha\|^{-1}\}\ll LQ^{-1}X^2.\end{align*}\end{lemma}
\begin{proof}For $\alpha\in \mathfrak{m}(Q)$, there exist $a$ and $q$ such that $1\le a\le q\le 2Q$, $(a,q)=1$ and
$|\alpha-a/q|\le 2Q(qX^2)^{-1}$. By a variant of Lemma 2.2 of
Vaughan \cite{V} (see also Exercise 2 in Chapter 2 \cite{V}), one
has
\begin{align*}\sum_{|x|\ll X}\min\{X,\ \|x\alpha\|^{-1}\}\ll
LX^2\Big(\frac{1}{q(1+X^2|\beta|)}+\frac{1}{X}+\frac{q(1+X^2|\beta|)}{X^2}\Big).\end{align*}
Since $\alpha\in \mathfrak{m}(Q)$, one has either $q>Q$ or
$|\alpha-a/q|>Q(qX^2)^{-1}$. Then the desired estimate follows
immediately.\end{proof}

\begin{lemma}\label{lemma43}Let $\alpha\in \mathfrak{m}$ and $\beta\in\R$. For $d\in \Q$, we define
\begin{align}\label{squareesitmate}f(\alpha,\,\beta)
=\sum_{1\le x\le X}\Lambda(x)e(\alpha d x^2+x\beta).\end{align}If
$d\not=0$, then one has
\begin{align*}f(\alpha,\,\beta) \ll
XL^{-K/5},\end{align*} where the implied constant depends only on
$d$ and $K$.
\end{lemma}
\begin{proof}The method used to handle $\sum_{1\le x\le X}\Lambda(x)e(\alpha  x^2)$ can be modified to establish
the desired conclusion. We only explain that the implied constant
is independent of $\beta$. By Vaughan's identity, we essentially
consider two types of exponential sums
\begin{align}\label{type1sum}\sum_{y}\eta_{y}\sum_{x}e(\alpha d x^2y^2+xy\beta)\end{align}
and
\begin{align}\label{type2sum}\sum_{x}\sum_{y}\xi_{x}\eta_ye(\alpha d x^2y^2+xy\beta).\end{align}
By Cauchy's inequality, to handle the summation (\ref{type2sum}),
it suffices to deal with
\begin{align*}\sum_{y_1}\sum_{y_2}\eta_{y_1}\overline{\eta_{y_2}}
\sum_{x}e\big(\alpha d
x^2(y_1^2-y_2^2)+x(y_1-y_2)\beta\big).\end{align*} One can apply
the differencing argument to the summation of the type
$\sum_{x}e(\alpha' x^2+x\beta')$ as follows
\begin{align*}\Big|\sum_{x}e(\alpha' x^2+x\beta')\Big|^2=&\sum_{x_1}\sum_{x_2}e\big(\alpha' (x_1^2-x_2^2)+(x_1-x_2)\beta'\big)
\\= & \sum_{h}\sum_{x}e(2\alpha'hx+h\beta')\ \le
\sum_{h}\Big|\sum_{x}e(2\alpha'hx)\Big|.\end{align*} This leads to
the fact that the estimate (\ref{squareesitmate}) is uniformly for
$\beta$.
\end{proof}

\begin{lemma}\label{lemma44}Let $\alpha\in \mathfrak{m}(Q)$. Suppose that $A$ is in the form
\begin{align} \label{specialA}A=\begin{pmatrix}A_1 & B & 0
\\ B^{T}  & A_2 & C \\ 0 & C^{T} & A_3 \end{pmatrix},\end{align}
where $\rank(B)\ge 3$ and $\rank(C)\ge 2$. Then we have
\begin{align}\label{eq45}S(\alpha)\ll X^n Q^{-5/2}L^{n+5/2}.\end{align}
\end{lemma}
\begin{remark}In view of the proof, the estimate (\ref{eq45}) still holds provided that $\rank(B)+\rank(C)\ge 5$. \end{remark}
\begin{proof}In view of (\ref{specialA}), we can write $S(\alpha)$ in the
form\begin{align*}S(\alpha)=&\sum_{\substack{1\le\mathbf{x}\le X\\
1\le\mathbf{y}\le X
\\ 1\le\mathbf{z}\le X}}\Lambda(\mathbf{x})
\Lambda(\mathbf{y})
\Lambda(\mathbf{z})e\Big(\alpha\big(\mathbf{x}^TA_1\mathbf{x}+2\mathbf{x}^TB\mathbf{y}+
\mathbf{y}^TA_2\mathbf{y}+2\mathbf{y}^TC\mathbf{z}+\mathbf{z}^TA_3\mathbf{z}\big)
\Big),\end{align*} where $\mathbf{x}\in \N^r$, $\mathbf{y}\in
\N^s$ and $\mathbf{z}\in \N^t$. Then we have
\begin{align*}S(\alpha) \le  L^s\sum_{
1\le\mathbf{y}\le X}&\Big| \sum_{\substack{1\le\mathbf{x}\le
X}}\Lambda(\mathbf{x})
e\Big(\alpha\big(\mathbf{x}^TA_1\mathbf{x}+2\mathbf{x}^TB\mathbf{y}\big)
\Big)\Big| \\ &\times \Big|\sum_{\substack{1\le\mathbf{z}\le X}}
\Lambda(\mathbf{z})e\Big(\alpha\big(2\mathbf{y}^TC\mathbf{z}+\mathbf{z}^TA_3\mathbf{z}\big)
\Big)\Big|.\end{align*} By Cauchy's inequality, we obtain
\begin{align}S(\alpha) \le & L^s\Big(\sum_{ 1\le\mathbf{y}\le X}\Big|
\sum_{\substack{1\le\mathbf{x}\le X}}\Lambda(\mathbf{x})
e\big(\alpha(\mathbf{x}^TA_1\mathbf{x}+2\mathbf{x}^TB\mathbf{y})
\big)\Big|^2\Big)^{1/2}\notag
\\ & \ \ \ \ \ \times \Big(\sum_{ 1\le\mathbf{y}\le X}\Big|\sum_{\substack{1\le\mathbf{z}\le X}}
\Lambda(\mathbf{z})e\big(\alpha(2\mathbf{y}^TC\mathbf{z}+\mathbf{z}^TA_3\mathbf{z})
\big)\Big|^2\Big)^{1/2}.\label{eq46}\end{align}We deduce by
opening the square that
\begin{align*}&\sum_{ 1\le\mathbf{y}\le X}\Big|
\sum_{\substack{1\le\mathbf{x}\le X}}\Lambda(\mathbf{x})
e\big(\alpha(\mathbf{x}^TA_1\mathbf{x}+2\mathbf{x}^TB\mathbf{y})
\big)\Big|^2\\ =&  \sum_{\substack{1\le\mathbf{x}_1\le
X}}\sum_{\substack{1\le\mathbf{x}_2\le
X}}\xi(\mathbf{x}_1,\mathbf{x}_2) \sum_{ 1\le\mathbf{y}\le
X}e\big(2\alpha(\mathbf{x}_1-\mathbf{x}_2)^TB\mathbf{y})\big)
\\ =&  \sum_{\substack{|\mathbf{h}|\le
X}}\sum_{\substack{1\le\mathbf{x}\le X
\\ 1\le\mathbf{x}+\mathbf{h}\le X}}\xi(\mathbf{x}+\mathbf{h},\mathbf{x}) \sum_{ 1\le\mathbf{y}\le
X}e\big(2\alpha(\mathbf{h}^TB\mathbf{y}) \big)
\\ \le & X^rL^{2r} \sum_{\substack{|\mathbf{h}|\le
X}}\Big|\sum_{ 1\le\mathbf{y}\le
X}e\big(2\alpha(\mathbf{h}^TB\mathbf{y}) \big)\Big|,\end{align*}
where $\xi(\mathbf{x}_1,\mathbf{x}_2)$ is defined as
\begin{align*}&\xi(\mathbf{x}_1,\mathbf{x}_2)=\Lambda(\mathbf{x}_1)\Lambda(\mathbf{x}_2)
e\big(\alpha(\mathbf{x}_1^TA_1\mathbf{x}_1-\mathbf{x}_2^TA_1\mathbf{x}_2)\big)
.\end{align*} We write \begin{align*} B=\begin{pmatrix}b_{1,1} &
\cdots & b_{1,s}
\\ \vdots & \cdots & \vdots
\\ b_{r,1} & \cdots & b_{r,s} \end{pmatrix}.\end{align*} Since $\rank(B)\ge 3$, without loss of
generality, we assume that $\rank(B_0)=3$, where
$B_0=(b_{i,j})_{1\le i,j\le 3}$. Let $B'=(b_{i,j})_{4\le i\le
r,1\le j\le 3}$. Then one has
\begin{align*}\sum_{\substack{|\mathbf{h}|\le
X}}\Big|\sum_{ 1\le\mathbf{y}\le
X}e\big(2\mathbf{h}^TB\mathbf{y}\alpha \big)\Big| & \le
X^{s-3}\sum_{|h_4|,\ldots,|h_r|\le X}
\sum_{\substack{|\mathbf{h}|\le X}}\Big|\sum_{ 1\le\mathbf{y}\le
X}e\big(2\alpha(\mathbf{h}^TB_0+\mathbf{k}^T)\mathbf{y}
\big)\Big|,\end{align*}where $\mathbf{k}^T=(h_4,\ldots,h_r)^TB'$.
By changing variables
$\mathbf{x}^T=2(\mathbf{h}^TB_0+\mathbf{k}^T)$, we obtain
\begin{align*}\sum_{\substack{|\mathbf{h}|\le
X}}\Big|\sum_{ 1\le\mathbf{y}\le
X}e\big(2\mathbf{h}^TB\mathbf{y}\alpha \big)\Big|  \le&
X^{s-3}\sum_{|h_4|,\ldots,|h_r|\le X}
\sum_{\substack{|\mathbf{x}|\ll X}}\Big|\sum_{ 1\le\mathbf{y}\le
X}e\big(\alpha(\mathbf{x}^T\mathbf{y}) \big)\Big| \\
\ll & X^{r+s-6} \sum_{\substack{|\mathbf{x}|\ll X}}\Big|\sum_{
1\le\mathbf{y}\le X}e\big(\alpha(\mathbf{x}^T\mathbf{y})
\big)\Big| .\end{align*}We apply Lemma \ref{lemma42} to conclude
that
\begin{align*}\sum_{\substack{|\mathbf{h}|\le
X}}\Big|\sum_{ 1\le\mathbf{y}\le
X}e\big(2\mathbf{h}^TB\mathbf{y}\alpha \big)\Big| \ll &
X^{r+s}Q^{-3}L^3 ,\end{align*} and therefore,
\begin{align}\label{eq47}\sum_{ 1\le\mathbf{y}\le X}\Big|
\sum_{\substack{1\le\mathbf{x}\le X}}\Lambda(\mathbf{x})
e\big(\alpha(\mathbf{x}^TA_1\mathbf{x}+2\mathbf{x}^TB\mathbf{y})
\big)\Big|^2\ll & X^{2r+s}Q^{-3}L^{2r+3} .\end{align} Similar to
(\ref{eq47}), we can prove
\begin{align}\label{eq48}\sum_{ 1\le\mathbf{y}\le X}\Big|\sum_{\substack{1\le\mathbf{z}\le X}}
\Lambda(\mathbf{z})e\big(\alpha(2\mathbf{y}^TC\mathbf{z}+\mathbf{z}^TA_3\mathbf{z})
\big)\Big|^2\ll & X^{2t+s}Q^{-2}L^{2t+2} .\end{align} The proof is
completed by invoking (\ref{eq46}), (\ref{eq47}) and (\ref{eq48}).
\end{proof}
\begin{lemma}\label{lemma45}Suppose that $A$ is in the form
(\ref{specialA}) with $\rank(B)\ge 3$ and $\rank(C)\ge 2$. Then we
have
\begin{align*}\int_{\mathfrak{m}}\big|S(\alpha)\big|d\alpha \ll X^{n-2}L^{-K/3}.\end{align*}
\end{lemma}
\begin{proof}By
Dirichlet's approximation theorem, for any $\alpha\in
[X^{-1},1+X^{-1}]$, there exist $a$ and $q$ with $1\le a\le q\le
X$ and $(a,q)=1$ such that $|\alpha-a/q|\le (qX)^{-1}$. Thus the
desired conclusion follows from Lemma \ref{lemma44} by the dyadic
argument.
\end{proof}
 \vskip3mm

\section{Quadratic forms with off-diagonal rank $\le 3$}

\begin{proposition}\label{proposition51}
Let $A$ be given by (\ref{defA}), and let $S(\alpha)$ be defined
in (\ref{defSalpha}). Suppose that $\rank(A)\ge 9$ and
$\rankoff(A)\le 3$. Then we
have\begin{align*}\int_{\mathfrak{m}}\big|S(\alpha)\big| d\alpha
\ll X^{n-2}L^{-K/6},\end{align*} where the implied constant
depends on $A$ and $K$.
\end{proposition}

Fron now on, we assume throughout Section 5 that $\rank(A)\ge 9$
and
\begin{align} \label{assumptionSec5}\rankoff(A)=\rank(B)=3 ,\end{align}
where
\begin{align} \label{defB}B=\begin{pmatrix}a_{1,4} & a_{1,5} & a_{1,6}
\\ a_{2,4} & a_{2,5} & a_{2,6}
\\ a_{3,4} & a_{3,5} & a_{3,6} \end{pmatrix}.\end{align}
Then we introduce $B_1,B_2,B_3\in M_{3,n-5}(\Z)$ defined as
\begin{align}\label{defB1} B_1=\begin{pmatrix}a_{1,5} & a_{1,6} & a_{1,7} & a_{1,8} & \cdots & a_{1,n}
\\ a_{2,5} & a_{2,6} & a_{2,7} & a_{2,8} & \cdots & a_{2,n}
\\ a_{3,5} & a_{3,6} & a_{3,7} & a_{3,8} & \cdots & a_{3,n}\end{pmatrix},\end{align}
\begin{align} \label{defB2}B_2=\begin{pmatrix}a_{1,4} & a_{1,6} & a_{1,7} & a_{1,8} & \cdots & a_{1,n}
\\ a_{2,4} & a_{2,6} & a_{2,7} & a_{2,8} & \cdots & a_{2,n}
\\ a_{3,4} & a_{3,6} & a_{3,7} & a_{3,8} & \cdots & a_{3,n}\end{pmatrix},\end{align}
and
\begin{align}\label{defB3} B_3=\begin{pmatrix}a_{1,4} & a_{1,5} & a_{1,7} & a_{1,8} & \cdots & a_{1,n}
\\ a_{2,4} & a_{2,5} & a_{2,7} & a_{2,8} & \cdots & a_{2,n}
\\ a_{3,4} & a_{3,5} & a_{3,7} & a_{3,8} & \cdots & a_{3,n}\end{pmatrix}.\end{align}
Subject to the assumption (\ref{assumptionSec5}), we have the
following.
\begin{lemma}\label{lemma51}If
$\rank(B_1)=\rank(B_2)=\rank(B_3)=2$, then one has
\begin{align*}\int_{\mathfrak{m}}\big|S(\alpha)\big| d\alpha \ll
X^{n-2}L^{-K/6}.\end{align*}
\end{lemma}

\begin{lemma}\label{lemma52}If
$\rank(B_1)=\rank(B_2)$ and $\rank(B_3)=3$, then one has
\begin{align*}\int_{\mathfrak{m}}\big|S(\alpha)\big| d\alpha \ll
X^{n-2}L^{-K/6}.\end{align*}
\end{lemma}

\begin{lemma}\label{lemma53}If
$\rank(B_1)=2$ and $\rank(B_2)=\rank(B_3)=3$, then one has
\begin{align*}\int_{\mathfrak{m}}\big|S(\alpha)\big| d\alpha \ll
X^{n-2}L^{-K/6}.\end{align*}
\end{lemma}

\begin{lemma}\label{lemma54}If
$\rank(B_1)=\rank(B_2)=\rank(B_3)=3$, then one has
\begin{align*}\int_{\mathfrak{m}}\big|S(\alpha)\big| d\alpha \ll
X^{n-2}L^{-K/6}.\end{align*}
\end{lemma}

\noindent\textit{Remark for the Proof of Proposition
\ref{proposition51}}. If $\rankoff(A)=0$, then $A$ is a diagonal
matrix and the conclusion is classical. When $\rankoff(A)=3$, our
conclusion follows from Lemmas \ref{lemma51}-\ref{lemma54}
immediately.  The method applied to establish Lemmas
\ref{lemma51}-\ref{lemma54} can be also used to deal with the case
$1\le \rankoff(A)\le 2$. Indeed, the proof of Proposition
\ref{proposition51} under the condition $1\le \rankoff(A)\le 2$ is
easier, and we omit the details. Therefore, our main task is to
establish Lemmas \ref{lemma51}-\ref{lemma54}.

\begin{lemma}\label{lemma55}Let $C\in M_{n,n}(\Q)$ be a symmetric matrix, and let
$H\in M_{n,k}(\Q)$. For $\alpha\in \R$ and $\boldsymbol{\beta}\in
\R^k$, we define
\begin{align*}\mathcal{F}(\alpha,\,\boldsymbol{\beta})
=\sum_{\mathbf{x}\in\, \mathcal{X}}w(\mathbf{x})e(\alpha
\mathbf{x}^{T}C\mathbf{x}+\mathbf{x}^{T}H\boldsymbol{\beta}),\end{align*}
where $\mathcal{X}\subset \Z^n$ is a finite subset of $\Z^n$. Then
we
have\begin{align*}\int_{[0,1]^{k+1}}\Big|\mathcal{F}(\alpha,\,\boldsymbol{\beta})\Big|^2d\alpha
d\boldsymbol{\beta} \ll \mathcal{N}(\mathcal{F}),\end{align*}
where $\mathcal{N}(\mathcal{F})$ is given by the following
\begin{align*}\mathcal{N}(\mathcal{F})=\sum_{\substack{\mathbf{x}\in
\,\mathcal{X},\,
\mathbf{y}\in\, \mathcal{X} \\
\mathbf{x}^{T}C\mathbf{x}=\mathbf{y}^{T}C\mathbf{y}
\\ \mathbf{x}^{T}H=\mathbf{y}^{T}H}}w(\mathbf{x})w(\mathbf{y})
.\end{align*}
\end{lemma}
\begin{proof}We can choose a natural number $h\in \N$ such that $hC\in M_{n,n}(\Z)$ and $hH\in M_{n,k}(\Z)$.
Then we deduce that
\begin{align*}&\int_{[0,1]^{k+1}}\big|\mathcal{F}(\alpha,\,\boldsymbol{\beta})\big|^2d\alpha
d\boldsymbol{\beta}
\\ \le & \int_{[0,h]^{k+1}}\Big|\sum_{\mathbf{x}\in\,
\mathcal{X}}w(\mathbf{x})e\Big(h^{-1}\alpha
\mathbf{x}^{T}(hC)\mathbf{x}+\mathbf{x}^{T}(hH)(h^{-1}\boldsymbol{\beta})\Big)\Big|^2d\alpha
d\boldsymbol{\beta}
\\=&h^{k+1}\int_{[0,1]^{k+1}}\Big|\sum_{\mathbf{x}\in\,
\mathcal{X}}w(\mathbf{x})e\Big(\alpha
\mathbf{x}^{T}(hC)\mathbf{x}+\mathbf{x}^{T}(hH)\boldsymbol{\beta}\Big)\Big|^2d\alpha
d\boldsymbol{\beta}.\end{align*}  By orthogonality, we have
\begin{align*}&\int_{[0,1]^{k+1}}\Big|\sum_{\mathbf{x}\in\,
\mathcal{X}}w(\mathbf{x})e\Big(\alpha
\mathbf{x}^{T}(hC)\mathbf{x}+\mathbf{x}^{T}(hH)\boldsymbol{\beta}\Big)\Big|^2d\alpha
d\boldsymbol{\beta} \\= & \sum_{\substack{\mathbf{x}\in
\,\mathcal{X},\
\mathbf{y}\in\, \mathcal{X} \\
\mathbf{x}^{T}(hC)\mathbf{x}=\mathbf{y}^{T}(hC)\mathbf{y}
\\ \mathbf{x}^{T}(hH)=\mathbf{y}^{T}(hH)}}w(\mathbf{x})w(\mathbf{y})
=\ \mathcal{N}(\mathcal{F}).\end{align*}Therefore, one obtains
have\begin{align*}\int_{[0,1]^{k+1}}\Big|\mathcal{F}(\alpha,\,\boldsymbol{\beta})\Big|^2d\alpha
d\boldsymbol{\beta} \le h^{k+1}
\mathcal{N}(\mathcal{F}),\end{align*}and this completes the proof.
\end{proof}
\begin{lemma}\label{lemma56}Let $C\in M_{n,n}(\Q)$ be a symmetric matrix, and let
$H\in M_{n,k}(\Q)$. We have\begin{align*}\mathcal{N}_1\le
\mathcal{N}_2,
\end{align*}where
\begin{align*}\mathcal{N}_1=\sum_{\substack{|\mathbf{x}|\ll X,\,
|\mathbf{y}|\ll X\\
\mathbf{x}^{T}C\mathbf{x}=\mathbf{y}^{T}C\mathbf{y}
\\ \mathbf{x}^{T}H=\mathbf{y}^{T}H}}1\ \
\textrm{ and }\ \ \mathcal{N}_2=\sum_{\substack{|\mathbf{x}|\ll
X,\,
|\mathbf{y}|\ll X \\
\mathbf{x}^{T}C\mathbf{y}=0
\\ \mathbf{x}^{T}H=0}}1
.\end{align*}
\end{lemma}
\begin{proof}By changing variables
$\mathbf{x}-\mathbf{y}=\mathbf{h}$ and
$\mathbf{x}+\mathbf{y}=\mathbf{z}$, the desired conclusion follows
immediately.
\end{proof}
The following result is well-known.
\begin{lemma}\label{lemma57}Let $C\in M_{k,m}(\Q)$. If $\rank(C) \geqslant 2$,
then one has
\begin{align*}\sum_{\substack{|\mathbf{x}|\ll X ,\ |\mathbf{y}|\ll X
 \\
\mathbf{x}^{T}C\mathbf{y}=0}}1\ll X^{k+m-2}L ,\end{align*} where
the implied constant depends on the matrix $C$.
\end{lemma}
 \vskip3mm

\subsection{Proof of Lemma \ref{lemma51}}

\begin{lemma}\label{lemma58}If
$\rank(B_1)=\rank(B_2)=\rank(B_3)=2$, then we can write $A$ in the
form
\begin{align}\label{Afirst} A=\begin{pmatrix}A_1 & B & 0
\\ B^{T}  & A_2 & C \\ 0 & C^{T} & D  \end{pmatrix},\end{align}
where $B\in GL_3(\Z)$, $C\in M_{3,n-6}(\Z)$ and
$D=\dia\{d_1,\ldots, d_{n-6}\}$ is a diagonal matrix.
\end{lemma}\begin{proof}
We write for $1\le j\le n-3$ that
\begin{align} \gamma_{j}=\begin{pmatrix}a_{1,\,3+j}
\\ a_{2,\,3+j} \\ a_{3,\,3+j}  \end{pmatrix}.\end{align}
Since $B=(\gamma_{1},\gamma_{2},\gamma_{3})\in GL_3(\Z)$,
$\gamma_{1}$, $\gamma_{2}$ and $\gamma_{3}$ are linearly
independent. For any $4\le j\le n-3$, one has
$\rank(\gamma_{2},\gamma_{3},\gamma_{j})\le \rank(B_1)=2$.
Therefore, we obtain $\gamma_{j}\in <\gamma_{2},\gamma_{3}>$.
Similarly, one has $\gamma_{j}\in <\gamma_{1},\gamma_{2}>$ and
$\gamma_{j}\in <\gamma_{1},\gamma_{2}>$. Then we can conclude that
$\gamma_{j}=0$ for $4\le j\le n-3$.

For $7\le i<j\le n$. we write
\begin{align*} B_{i,j}=\begin{pmatrix}a_{1,4} & a_{1,5} &
a_{1,6} & a_{1,j}
\\ a_{2,4} & a_{2,5} &
a_{2,6} & a_{2,j} \\ a_{3,4} & a_{3,5} & a_{3,6} & a_{3,j}
\\ a_{i,4} & a_{i,5} &
a_{i,6} & a_{i,j}
\end{pmatrix}=\begin{pmatrix}\eta_1^{T}
\\ \eta_2^{T} \\ \eta_3^{T}
\\ \eta_4^{T}
\end{pmatrix}.\end{align*}
Since $3\le \rank(B_{i,j})\le \rankoff(A)=3$, we conclude that
$\eta_4^{T}$ can be linearly represented by $\eta_1^{T}$,
$\eta_2^{T}$ and $\eta_3^{T}$. Then we obtain $a_{i,j}=0$ due to
$a_{1,j}=a_{2,j}=a_{3,j}=0$. Therefore, the matrix $A$ is in the
form (\ref{Afirst}). We complete the proof.
\end{proof}

\noindent{\textit{Proof of Lemma \ref{lemma51}}}. By Lemma
\ref{lemma58}, we have
\begin{align*} S(\alpha)=& \sum_{\substack{\mathbf{x}\in \N^3
\\ 1\le \mathbf{x}\le X}} \sum_{\substack{\mathbf{y}\in \N^3
\\ 1\le \mathbf{y}\le X}} \sum_{\substack{\mathbf{z}\in
\N^{n-6}\\ 1\le \mathbf{z}\le
X}}e\Big(\alpha(\mathbf{x}^{T}A_1\mathbf{x}+2\mathbf{x}^{T}B\mathbf{y}+
\mathbf{y}^{T}A_2\mathbf{y}+2\mathbf{z}^{T}C^{T}\mathbf{y}+\mathbf{z}^{T}D\mathbf{z})\Big)
\\ & \ \ \ \ \ \ \ \ \ \ \ \ \ \ \ \ \ \ \ \ \ \times\Lambda(\mathbf{x})
\Lambda(\mathbf{y}) \Lambda(\mathbf{z}).\end{align*} By
orthogonality, we have
\begin{align*} S(\alpha)=& \int_{[0,1]^3}\sum_{\substack{\mathbf{w}\in \Z^3
\\ |\mathbf{w}|\ll X}}\sum_{\substack{\mathbf{x}\in \N^3
\\ 1\le \mathbf{x}\le X}} \sum_{\substack{\mathbf{y}\in \N^3
\\ 1\le \mathbf{y}\le X}} \sum_{\substack{\mathbf{z}\in
\N^{n-6}\\ 1\le \mathbf{z}\le
X}}e\big(\alpha(\mathbf{x}^{T}A_1\mathbf{x}+\mathbf{w}^{T}\mathbf{y}+\mathbf{z}^{T}D\mathbf{z})\big)
\\ & \ \ \ \ \ \ \ \ \ \ \ \ \ \ \ \ \ \times
e\Big((2\mathbf{x}^{T}B+
\mathbf{y}^{T}A_2+2\mathbf{z}^{T}C^{T}-\mathbf{w}^{T})\boldsymbol{\beta}\Big)\Lambda(\mathbf{x})
\Lambda(\mathbf{y}) \Lambda(\mathbf{z})d
\boldsymbol{\beta},\end{align*} where
$\boldsymbol{\beta}=(\beta_1,\beta_2,\beta_3)^T$ and we use
$d\boldsymbol{\beta}$ to denote $d\beta_1d\beta_2d\beta_3$. We
define
\begin{align*} \mathcal{F}(\alpha,\boldsymbol{\beta})=& \sum_{\substack{\mathbf{x}\in \N^3
\\ 1\le \mathbf{x}\le X}}
e\big(\alpha\mathbf{x}^{T}A_1\mathbf{x}+2\mathbf{x}^{T}B\boldsymbol{\beta}\big)\Lambda(\mathbf{x}),\end{align*}
and
\begin{align*} f_j(\alpha,\boldsymbol{\beta})=& \sum_{\substack{1\le z\le X}} e\big(\alpha d_j
z^2 +2z\xi_j^{T}\boldsymbol{\beta}\big)\Lambda(z),\end{align*}
where $\xi_j=(a_{4,6+j},a_{5,6+j},a_{6,6+j})^T$ for $1\le j\le
n-6$. On writing $I_3=(\mathbf{e}_1,\mathbf{e}_2,\mathbf{e}_3)$,
we introduce\begin{align*}
\mathcal{H}_j(\alpha,\boldsymbol{\beta})=& \sum_{|w|\ll
X}\sum_{\substack{1\le y\le X}} e\big(\alpha wy
+y\gamma_j^{T}\boldsymbol{\beta}-w\mathbf{e}_j^T\boldsymbol{\beta}\big)\Lambda(y),\end{align*}
where $\gamma_j^{T}=(a_{3+j,4},a_{3+j,5},a_{3+j,6})$ for $1\le
j\le 3$. We conclude from above
\begin{align}\label{intcase1} \int_{\mathfrak{m}}\big|S(\alpha)\big|d\alpha\le & \int_{\mathfrak{m}}\int_{[0,1]^3}
\Big|\mathcal{F}(\alpha,\boldsymbol{\beta})\mathcal{H}_1(\alpha,\boldsymbol{\beta})
\mathcal{H}_2(\alpha,\boldsymbol{\beta})
\mathcal{H}_3(\alpha,\boldsymbol{\beta})\prod_{j=1}^{n-6}
f_j(\alpha,\boldsymbol{\beta})\Big|d\boldsymbol{\beta}d\alpha.\end{align}
We first consider the case $\rank(D)\geqslant 3$. Without loss of
generality, we assume $d_1d_2d_3\not=0$. By (\ref{intcase1}) and
the Cauchy-Schwarz inequality, one has
\begin{align} \label{eq59}\int_{\mathfrak{m}}\big|S(\alpha)\big|d\alpha \le
\mathcal{I}_1^{1/2}\mathcal{I}_2^{1/2}\sup_{\substack{\alpha\in\mathfrak{m}\\
\boldsymbol{\beta}\in [0,1]^3}}\Big
|\prod_{j=3}^{n-6}f_j(\alpha,\boldsymbol{\beta})\Big|,\end{align}
where
\begin{align} \label{eq512}\mathcal{I}_1=\int_{[0,1]^4}
\Big|\mathcal{F}(\alpha,\boldsymbol{\beta})f_1(\alpha,\boldsymbol{\beta})f_2(\alpha,\boldsymbol{\beta})\Big|^2d\boldsymbol{\beta}d\alpha\end{align}
and
\begin{align}\label{eq511}\mathcal{I}_2=\int_{[0,1]^4} \Big|\mathcal{H}_1(\alpha,\boldsymbol{\beta})
\mathcal{H}_2(\alpha,\boldsymbol{\beta})
\mathcal{H}_3(\alpha,\boldsymbol{\beta})\Big|^2d\boldsymbol{\beta}d\alpha.\end{align}
By Lemmas \ref{lemma55} and \ref{lemma56}, one has
\begin{align*} \mathcal{I}_1\ \ll &\ L^{10}\sum_{\substack{
|\mathbf{x}|,|\mathbf{x}'|,|z_1|,|z_1'|,|z_2|,|z_2'|\ll X
\\
\mathbf{x}^{T}A_1\mathbf{x}+d_1z_1^2+d_2z_2^2=\mathbf{x}'^{T}A_1\mathbf{x}'+d_1z_1'^2+d_2z_2'^2
\\ \mathbf{x}^{T}B+ z_1 \gamma_1^{T} +z_2 \gamma_2^T =
\mathbf{x}'^{T}B+ z_1' \gamma_1^{T} +z_2' \gamma_2^T  }}1
\\  \ll  &\ L^{10}\sum_{\substack{
|\mathbf{x}|,|\mathbf{x}'|,|z_1|,|z_1'|,|z_2|,|z_2'|\ll X
\\ \mathbf{x}^{T}A_1\mathbf{x}'+d_1z_1z_1'+d_2z_2z_2'=0
\\ \mathbf{x}^{T}B+ z_1 \gamma_1^{T} +z_2 \gamma_2^T =0 }}1.\end{align*}
Since $B$ is invertible, we obtain
\begin{align*} \mathcal{I}_1\ \ll &\ L^{10}\sum_{\substack{
|\mathbf{x}'|,|z_1|,|z_1'|,|z_2|,|z_2'|\ll X
\\ (z_1 \gamma_1^{T} +z_2 \gamma_2^T)B^{-1}A_1\mathbf{x}'+d_1z_1z_1'+d_2z_2z_2'=0
}}1.\end{align*}Then we conclude from Lemma \ref{lemma57} that
\begin{align}\label{boundI1-1}  \mathcal{I}_1\ \ll &\
X^5L^{11}.\end{align} Similar to (\ref{boundI1-1}), one can deduce
from Lemmas \ref{lemma55}-\ref{lemma57},
\begin{align} \label{boundI2-1}\mathcal{I}_2\ \ll\
X^{7}L^{7}.\end{align}Since $d_3\not=0$, we obtain by Lemma
\ref{lemma43}
\begin{align*}
\sup_{\substack{\alpha\in\mathfrak{m}\\ \boldsymbol{\beta}\in
[0,1]^3}} \Big|f_3(\alpha,\boldsymbol{\beta})\Big| \ll
XL^{-K/5},\end{align*} and thereby
\begin{align}\label{boundsup1}
\sup_{\substack{\alpha\in\mathfrak{m}\\ \boldsymbol{\beta}\in
[0,1]^3}}
\Big|\prod_{j=3}^{n-6}f_j(\alpha,\boldsymbol{\beta})\Big| \ll
X^{n-8}L^{-K/5}.\end{align}Now we conclude from (\ref{eq59}),
(\ref{boundI1-1}), (\ref{boundI2-1}) and (\ref{boundsup1}) that
\begin{align*} \int_{\mathfrak{m}}\big|S(\alpha)\big|d\alpha \ll X^{n-2}L^{-K/6} .\end{align*}

 Next we consider the case $1\le \rank(D)\le 2$. Without loss of generality, we suppose
that $d_1\not=0$ and $d_k=0$ for $3\le k\le n$. Since $\rank(A)\ge
9$, there exists $k$ with $3\le k\le n-6$ such that
$\xi_k\not=0\in\Z^3$. Then we can find $i,j$ with $1\le i<j\le 3$
so that $\rank(\mathbf{e}_i, \mathbf{e}_j,\xi_{k})=3$. Without
loss of generality, we can assume that $i=1,j=2$ and $k=3$. One
has
\begin{align*} \int_{\mathfrak{m}}\big|S(\alpha)\big|d\alpha \le
\sup_{\substack{\alpha\in\mathfrak{m}\\ \boldsymbol{\beta}\in
[0,1]^3}} \Big|\prod_{j\not=3}f_j(\alpha,\boldsymbol{\beta})\Big|
&\Big(\int_{[0,1]^4} \Big|\mathcal{F}(\alpha,\boldsymbol{\beta})
\mathcal{H}_3(\alpha,\boldsymbol{\beta})\Big|^2d\boldsymbol{\beta}d\alpha\Big)^{1/2}
\\ \times &\Big(\int_{[0,1]^4}\Big |\mathcal{H}_1(\alpha,\boldsymbol{\beta})
\mathcal{H}_2(\alpha,\boldsymbol{\beta})
f_3(\alpha,\boldsymbol{\beta})\Big|^2d\boldsymbol{\beta}d\alpha\Big)^{1/2}.\end{align*}
Then we can apply Lemmas \ref{lemma55}-\ref{lemma57} to deduce
\begin{align*} \int_{[0,1]^4}
\Big|\mathcal{F}(\alpha,\boldsymbol{\beta})
\mathcal{H}_3(\alpha,\boldsymbol{\beta})\Big|^2d\boldsymbol{\beta}d\alpha\ll
X^5L^9\end{align*} and
\begin{align*} \int_{[0,1]^4} \Big|\mathcal{H}_1(\alpha,\boldsymbol{\beta})
\mathcal{H}_2(\alpha,\boldsymbol{\beta})
f_3(\alpha,\boldsymbol{\beta})\Big|^2d\boldsymbol{\beta}d\alpha\ll
X^5L^7.\end{align*} It follows from Lemma \ref{lemma43} that
\begin{align*}
\sup_{\substack{\alpha\in\mathfrak{m}\\ \boldsymbol{\beta}\in
[0,1]^3}} \Big|\prod_{j\not=3}f_j(\alpha,\boldsymbol{\beta})\Big|
\ll X^{n-7}L^{-K/5}.\end{align*}Then we obtain from above again
\begin{align*} \int_{\mathfrak{m}}\big|S(\alpha)\big|d\alpha \ll X^{n-2}L^{-K/6} .\end{align*}

Now it suffices to assume $D=0$. Then the matrix $A$ is in the
form\begin{align}\label{Aform} A=\begin{pmatrix}A_1 & B & 0
\\ B^{T}  & A_2 & C \\ 0 & C^{T} & 0  \end{pmatrix}.\end{align}
It follows from $\rank(A)\ge 9$ that $\rank(C)\ge 3$. By Lemma
\ref{lemma45}, \begin{align*}
\int_{\mathfrak{m}}\big|S(\alpha)\big|d\alpha \ll X^{n-2}L^{-K/3}
.\end{align*}The proof of Lemma \ref{lemma51} is completed.
 \vskip3mm

\subsection{Proof of Lemma \ref{lemma52}}

\begin{lemma}\label{lemma59}If
$\rank(B_1)=\rank(B_2)=2$ and $\rank(B_3)=3$, then the symmetric
integral matrix $A$ can be written in the form
\begin{align*} A=\begin{pmatrix}A_1 & C & \gamma_3\xi^{T}
\\ C^{T} & A_2 & V  \\ \xi\gamma_3^{T} & V^{T} & D+h\xi\xi^{T}  \end{pmatrix},\end{align*}
where $C=(\gamma_1,\gamma_2)\in M_{3,2}(\Z)$, $\gamma_3\in
\Q^{3}$, $\xi\in \Z^{n-6}$, $V\in M_{2,n-6}(\Z)$, $h\in \Q$ and
$D=\dia\{d_1,\ldots, d_{n-6}\}\in M_{n-5,n-5}(\Q)$ is a diagonal
matrix. Moreover, one has $(\gamma_1,\gamma_2,\gamma_3)\in
GL_3(\Q)$.
\end{lemma}\begin{proof}
Let us write
\begin{align*} \gamma_{j}'=\begin{pmatrix}a_{1,\,3+j}
\\ a_{2,\,3+j} \\ a_{3,\,3+j}  \end{pmatrix}\ \textrm{ for }\ 1\le j\le n-3.\end{align*}
Since $\rank(\gamma_{1}',\gamma_{2}',\gamma_{3}')=\rank(B)=3$, we
conclude that $\gamma_{1}'$, $\gamma_{2}'$ and $\gamma_{3}'$ are
linearly independent. For any $4\le j\le n-3$, we deduce from
$\rank(B_1)=\rank(B_2)=2$ that $\gamma_{j}'\in
<\gamma_{2}',\gamma_{3}'>\cap<\gamma_{2}',\gamma_{3}'>=<\gamma_3'>$.
Therefore, we can write $A$ in the form
\begin{align*} A=\begin{pmatrix}A_1 & C & \gamma_3\xi^{T}
\\ C^{T} & A_2 & V  \\ \xi\gamma_3^{T} & V^{T} & A_3  \end{pmatrix},\end{align*}
where $C=(\gamma_1,\gamma_2)\in M_{3,2}(\Z)$, $\gamma_3\in
\Q^{3}$, $\xi\in \Z^{n-6}$, $V\in M_{2,n-6}(\Z)$ and $A_3\in
M_{n-5,n-5}(\Q)$.

For $6\le j\le n$. we define
$\eta_j^{T}=(a_{j,4},\ldots,a_{j,j-1},a_{j,j+1},\ldots,a_{j,n})^{T}\in
\Z^{n-4}$. Then we set
$\theta_{i,j}^{T}=(a_{i,4},\ldots,a_{i,j-1},a_{i,j+1},\ldots,a_{i,n})^{T}\in
\Z^{n-4}$ for $1\le i\le 3$. Since
$\rankoff(A)=\rank(B)=\rank(B_3)=3$, $\eta_j$ can be linearly
represented by $\theta_{1,j}$, $\theta_{2,j}$ and $\theta_{3,j}$.
Let\begin{align*} \theta_{i}^{T}=(a_{i,4},\ldots,a_{i,n})^{T}\in
\Z^{n-3}\ \textrm{ for } \ 1\le i\le 3.\end{align*} Then one can
choose $a_{j,j}'\in \Q$ such that
$(a_{j,4},\ldots,a_{j,j-1},a_{j,j}',a_{j,j+1},\ldots,a_{j,n})$ is
linearly represented by $\theta_{1}$, $\theta_{2}$ and
$\theta_{3}$. We consider $A_3$ and $A_3'$ defined as
\begin{align*}A_{3}=\begin{pmatrix}a_{6,6} & \cdots &
a_{6,n}
\\ \vdots & \cdots & \vdots
\\ a_{n,6} & \cdots &
a_{n,n}
\end{pmatrix} \ \textrm{ and }\  A_{3}'=\begin{pmatrix}a_{6,6}' & \cdots &
a_{6,n}'
\\ \vdots & \cdots & \vdots
\\ a_{n,6}' & \cdots &
a_{n,n}'
\end{pmatrix},\end{align*}where $a_{i,j}'=a_{i,j}$ for $6\le i\not=j\le
n$. Since $A_3'$ is symmetric, we conclude from above that
$A_3'=h\xi\xi^{T}$ for some $h\in\Q$. The proof is completed by
noting that $D=A_3-A_3'$ is a diagonal matrix. \end{proof}

\noindent{\textit{Proof of Lemma \ref{lemma52}}}. One can deduce
from Lemma \ref{lemma59} that
\begin{align*} S(\alpha)=& \sum_{\substack{\mathbf{x}\in \N^3
\\ 1\le \mathbf{x}\le X}} \sum_{\substack{\mathbf{y}\in \N^2
\\ 1\le \mathbf{y}\le X}} \sum_{\substack{\mathbf{z}\in
\N^{n-5}\\ 1\le \mathbf{z}\le
X}}e\Big(\alpha(\mathbf{x}^{T}A_1\mathbf{x}+2\mathbf{x}^T\gamma_3\xi^{T}\mathbf{z}
+\mathbf{z}^{T}D\mathbf{z}+
h\mathbf{z}^{T}\xi\xi^{T}\mathbf{z})\Big)
\\ & \ \ \ \ \ \ \ \ \ \ \ \ \ \ \ \ \ \ \ \ \ \times
e\Big(\alpha(2\mathbf{x}^{T}C\mathbf{y}
+\mathbf{y}^{T}A_2\mathbf{y}+2\mathbf{z}^{T}C^{T}\mathbf{y})\Big)\Lambda(\mathbf{x})
\Lambda(\mathbf{y}) \Lambda(\mathbf{z}).\end{align*} We introduce
new variables $\mathbf{w}\in \Z^2$ and $s\in \Z$ to replace
$2\mathbf{x}^{T}C +\mathbf{y}^{T}A_2+2\mathbf{z}^{T}C^{T}$ and
$\xi^{T}\mathbf{z}$, respectively. Therefore, we have
\begin{align*} S(\alpha)=& \int_{[0,1]^3}\sum_{|s|\ll X}\sum_{\substack{\mathbf{w}\in
\Z^2
\\ |\mathbf{w}|\ll X}}\sum_{\substack{\mathbf{x}\in \N^3
\\ 1\le \mathbf{x}\le X}} \sum_{\substack{\mathbf{y}\in \N^2
\\ 1\le \mathbf{y}\le X}} \sum_{\substack{\mathbf{z}\in
\N^{n-5}\\ 1\le \mathbf{z}\le X}}\Lambda(\mathbf{x})
\Lambda(\mathbf{y}) \Lambda(\mathbf{z})
\\ & \ \ \ \ \ \ \ \ \ \ \ \ \ \ \ \ \ \ \ \ \ \times
e\Big(\alpha(\mathbf{x}^{T}A_1\mathbf{x}+\mathbf{w}^{T}\mathbf{y}+\mathbf{z}^{T}D\mathbf{z}+
2\mathbf{x}^T\gamma_3s+hs^2)\Big)
\\ & \ \ \ \ \ \ \ \ \ \ \ \ \ \ \ \ \ \ \ \ \ \times
e\Big((2\mathbf{x}^{T}C+
\mathbf{y}^{T}A_2+2\mathbf{z}^{T}V^{T}-\mathbf{w}^{T})\boldsymbol{\beta}'\Big)
\\ & \ \ \ \ \ \ \ \ \ \ \ \ \ \ \ \ \ \ \ \ \ \times
e\Big((\xi^{T}\mathbf{z}-s)\beta_3\Big)d\boldsymbol{\beta},\end{align*}
where $\boldsymbol{\beta}'=(\beta_1,\beta_2)^T$,
$\boldsymbol{\beta}=(\beta_1,\beta_2, \beta_3)^T$  and
$d\boldsymbol{\beta}=d\beta_1d\beta_2d\beta_3$. We define
\begin{align*} \mathcal{F}(\alpha,\boldsymbol{\beta})=&
\sum_{|s|\ll X}\sum_{\substack{\mathbf{x}\in \N^3
\\ 1\le \mathbf{x}\le X}}
e\Big(\alpha(\mathbf{x}^{T}A_1\mathbf{x}+2\mathbf{x}^T\gamma_3s+hs^2)
+2\mathbf{x}^{T}C\boldsymbol{\beta}'-s\beta_3\Big)\Lambda(\mathbf{x}).\end{align*}
On writing $I_2=(\mathbf{e}_1,\mathbf{e}_2)$, we
introduce\begin{align*} \mathcal{H}_j(\alpha,\boldsymbol{\beta})=&
\sum_{|w|\ll X}\sum_{\substack{1\le y\le X}} e(\alpha wy
+y\rho_j^{T}\boldsymbol{\beta}'-w\mathbf{e}_j^T\boldsymbol{\beta}')\Lambda(y),\end{align*}
where $\rho_j=(a_{3+j,4},a_{3+j,5})^{T}$ for $1\le j\le 2$. Let
$\xi=(\epsilon_1,\ldots,\epsilon_{n-5})^{T}$. Then we define
\begin{align*} f_j(\alpha,\boldsymbol{\beta})=& \sum_{\substack{1\le z\le X}} e(\alpha d_j
z^2
+2z\upsilon_j^{T}\boldsymbol{\beta}'+\epsilon_jz\beta_3)\Lambda(z),\end{align*}
where $V=(\upsilon_1, \ldots,\upsilon_{n-5})$ and
$\xi_j=(a_{4,5+j},a_{5,5+j})^{T}$ for $1\le j\le n-5$. With above
notations, we obtain
\begin{align}\label{intcase2} \int_{\mathfrak{m}}\big|S(\alpha)\big|d\alpha \le & \int_{\mathfrak{m}}\int_{[0,1]^3}
\Big|\mathcal{F}(\alpha,\boldsymbol{\beta})\mathcal{H}_1(\alpha,\boldsymbol{\beta})
\mathcal{H}_2(\alpha,\boldsymbol{\beta}) \prod_{j=1}^{n-5}
f_j(\alpha,\boldsymbol{\beta})\Big|d\boldsymbol{\beta}d\alpha.\end{align}
Let
\begin{align*} \mathcal{I}_1=\int_{[0,1]^4}
\Big|\mathcal{F}(\alpha,\boldsymbol{\beta})f_i(\alpha,\boldsymbol{\beta})\Big|^2d\boldsymbol{\beta}d\alpha.\end{align*}
and
\begin{align*} \mathcal{I}_2=\int_{[0,1]^4} \Big|\mathcal{H}_1(\alpha,\boldsymbol{\beta})
\mathcal{H}_2(\alpha,\boldsymbol{\beta})
f_j(\alpha,\boldsymbol{\beta})\Big|^2d\boldsymbol{\beta}d\alpha.\end{align*}
By (\ref{intcase2}) and the Cauchy-Schwarz inequality, one has for
$i\not=j$ that
\begin{align}\label{eq2-I1I2} \int_{\mathfrak{m}}\big|S(\alpha)\big|d\alpha \le &
\mathcal{I}_1^{1/2}\mathcal{I}_2^{1/2} \sup_{\substack{\alpha\in\mathfrak{m}\\
\boldsymbol{\beta}\in [0,1]^3}}
\Big|\prod_{k\not=i,j}f_k(\alpha,\boldsymbol{\beta})\Big|.\end{align}
One can deduce by Lemmas \ref{lemma55} and \ref{lemma56} that
\begin{align*} \mathcal{I}_1
\ll L^{8}\sum_{\substack{
|\mathbf{x}|,|\mathbf{x}'|,|s|,|s'|,|z|,|z'|\ll X
\\ \mathbf{x}^{T}A_1\mathbf{x}'+2\mathbf{x}^{T}\gamma_3s'+2s\gamma_3^{T}\mathbf{x}'+hss'+d_izz'=0
\\ \mathbf{x}^{T}C+ z \upsilon_i^{T}  =0
\\  s=\epsilon_i z}}1.\end{align*} Note that
\begin{align*} \sum_{\substack{
|\mathbf{x}|,|\mathbf{x}'|,|s|,|s'|,|z|,|z'|\ll X
\\ \mathbf{x}^{T}A_1\mathbf{x}'+2\mathbf{x}^{T}\gamma_3s'+2s\gamma_3^{T}\mathbf{x}'+hss'+d_izz'=0
\\ \mathbf{x}^{T}C+ z \upsilon_i^{T}  =0
\\  s=\epsilon_i z}}1=&
\sum_{\substack{ |\mathbf{x}|,|\mathbf{x}'|,|s'|,|z|,|z'|\ll X
\\ \mathbf{x}^{T}A_1\mathbf{x}'+2\mathbf{x}^{T}\gamma_3s'+2\epsilon_i z\gamma_3^{T}\mathbf{x}'+
h\epsilon_i zs'+d_izz'=0
\\ \mathbf{x}^{T}C+ z \upsilon_i^{T}  =0}}1
\\ =&
\sum_{\substack{ |\mathbf{x}|,|\mathbf{x}'|,|s|,|s'|,|z|,|z'|\ll X
\\ \mathbf{x}^{T}A_1\mathbf{x}'+2ss'+2\epsilon_i z\gamma_3^{T}\mathbf{x}'+
h\epsilon_i zs'+d_izz'=0
\\ \mathbf{x}^{T}(C,\gamma_3)+ (z \upsilon_i^{T},-s)  =0}}1.\end{align*}
Recalling $\rank(C,\gamma_3)=3$, one can replace $\mathbf{x}$ by $
(z \upsilon_i^{T},-s)(C,\gamma_3)^{-1}$. Therefore, by Lemma
\ref{lemma57}, one has
\begin{align}\label{boundI1-2}\mathcal{I}_1
\ll X^5L^{9}\ \ \textrm{ if } \ \ d_i\not=0\ \textrm{ or } \
\epsilon_i\not=0.\end{align} Similarly, one can deduce from Lemmas
\ref{lemma55}-\ref{lemma57} that
\begin{align}\label{boundI2-2}\mathcal{I}_2 \ll
X^5L^5 \ \ \textrm{ if } \ \ \epsilon_j\not=0 .\end{align}

Since $\rank(B_3)=3$, one has $\epsilon_l\not=0$ for some $l$
satisfying $2\le l\le n-5$. We may assume $\epsilon_2\not=0$. We
also have $\epsilon_1\not=0$ due to $\rank(B)=3$. Since
$\rank(D)+\rank(V)+1+5\geqslant \rank(A)\geqslant 9$, we obtain
$\rank(D)\geqslant 1$. Therefore, if $d_l\not=0$ or
$\epsilon_l\not=0$ for some $l\ge 3$, then we can find $i,j,k$
pairwise distinct so that $(d_i,\epsilon_i)^{T}\not=0\in \Q^2$,
$\epsilon_j\not=0$ and $d_k\not=0$. Then we conclude from
(\ref{eq2-I1I2}), (\ref{boundI1-2}) and (\ref{boundI2-2}) that
\begin{align*}
\int_{\mathfrak{m}}\big|S(\alpha)\big|d\alpha \ll X^{n-2}L^{-K/6}
.\end{align*}

Next we assume $d_l=\epsilon_l=0$ for all $l\ge 3$. Then we can
represent $A$ in the form
\begin{align*} A=\begin{pmatrix}A_1 & H & 0
\\ H^{T} & A_2 & W  \\ 0 & W^{T} & 0  \end{pmatrix},\end{align*}
where $H\in M_{3,4}(\Z)$ and $W\in M_{2,n-7}(\Z)$.  It follows
from $\rank(B)=3$ and $\rank{A}\ge 9$ that $\rank(H)\ge 3$ and
$\rank(W)\ge 2$. We apply Lemma \ref{lemma45} to conclude
\begin{align*}
\int_{\mathfrak{m}}\big|S(\alpha)\big|d\alpha \ll X^{n-2}L^{-K/3}
.\end{align*} The proof of Lemma \ref{lemma52} is completed.
 \vskip3mm

\subsection{Proof of Lemma \ref{lemma53}}

The proof of Lemma \ref{lemma59} can be modified to establish the
following result. The detail of the proof is omitted.
\begin{lemma}\label{lemma510}If
$\rank(B_1)=2$ and $\rank(B_2)=\rank(B_3)=3$, then we can write
$A$ in the form
\begin{align}\label{Athird} A=\begin{pmatrix}A_1 & \gamma_1 & (\gamma_2,\gamma_3)C
\\ \gamma_1^{T} & a & \upsilon^{T}
\\ C^{T}(\gamma_2,\gamma_3)^{T} & \upsilon & D+C^{T}HC  \end{pmatrix},\end{align}
where $\gamma_1\in\Z^3$, $\gamma_2,\gamma_3\in \Q^3$, $C\in
M_{2,n-4}(\Z)$, $a\in \Z$, $\upsilon\in \Z^{n-4}$, $H\in
M_{2,2}(\Q)$ and $D=\dia\{d_1,\ldots, d_{n-4}\}\in
M_{n-4,n-4}(\Q)$ is a diagonal matrix. Moreover, one has
$(\gamma_1,\gamma_2,\gamma_3)\in GL_3(\Q)$.
\end{lemma}

\begin{lemma}\label{lemma511} Let $A$ be given by (\ref{Athird}).
We write
\begin{align}\label{notationCv}C=(\xi_1,\ldots,\xi_{n-4})\
\textrm{ and }\ \upsilon^{T}=(v_1,\ldots,v_{n-4}).
\end{align}Let
\begin{align}\label{Rijk} R_{i,j,k}=\begin{pmatrix}\xi_i & \xi_j &
\xi_k
\\ v_i & v_j & v_k   \end{pmatrix}.\end{align}Under the conditions
in Lemma \ref{lemma510}, one can find pairwise distinct $i,j,k,u$
with $1\le i,j,k,u\le n-4$ such that at least one of the following
two statements holds: (i) $\rank(R_{i,j,k})=3$ and $d_u\not=0$;
(ii) $\rank(\xi_i,\xi_j)=2$, $d_kd_u\not=0$.
\end{lemma}
\begin{proof}It follows from $9\le \rank(A)\le
\rank(D)+\rank(V)+\rank(C)+4$ that $\rank(D)\ge 2$. If $\rank(D)=
2$, say $d_1d_2\not=0$, then $\rank(R)\ge 3$, where
\begin{align*} R=\begin{pmatrix}\xi_1 & \cdots &
\xi_{n-4}
\\ v_1 & \cdots & v_{n-4}   \end{pmatrix}.\end{align*}
Then statement (i) holds. Next we assume $\rank(D)\ge 3$. Note
that $\rank(\xi_1,\xi_2)=2$ due to $\rank(B)=3$. If $d_rd_s\not=0$
for some $r>s\ge 3$, then statement (i)  follows by choosing
$i=1,j=2,k=r$ and $u=s$. Therefore, we now assume that
$\rank(D)=3$ and $d_1d_2\not=0$. Without loss of generality, we
suppose that $d_3\not=0$ and $d_s=0(4\le s\le n-4)$ . We consider
$\rank(\xi_1,\xi_s)$ and $\rank(\xi_2,\xi_s)$ for $4\le s\le n-4$.
If $\rank(\xi_1,\xi_s)=2$ for some $s$ with $4\le s\le n-4$, then
one can choose $i=1$, $j=s$ and $k=3$ to establish statement (ii).
Similarly, statement (ii) follows if $\rank(\xi_2,\xi_s)=2$ for
some $s$ with $4\le s\le n-4$. Thus it remains to consider the
case $\rank(\xi_1,\xi_s)=\rank(\xi_2,\xi_s)=1$ for $4\le s\le
n-4$. However, it follows from
$\rank(\xi_1,\xi_2)=\rank(\xi_1,\xi_2,\xi_s)=2$ that $\xi_s=0$,
and this is contradictory to the condition $\rank(A)\ge 9$. We
complete the proof of Lemma \ref{lemma511}.
\end{proof}

\noindent{\textit{Proof of Lemma \ref{lemma53}}}. We deduce from
Lemma \ref{lemma510} that
\begin{align*} S(\alpha)=& \sum_{\substack{\mathbf{x}\in \N^3
\\ 1\le \mathbf{x}\le X}} \sum_{\substack{1\le y\le X}} \sum_{\substack{\mathbf{z}\in
\N^{n-4}\\ 1\le \mathbf{z}\le
X}}e\Big(\alpha\big(\mathbf{x}^{T}A_1\mathbf{x}+2\mathbf{x}^T(\gamma_1,\gamma_2)C\mathbf{z}
+\mathbf{z}^{T}D\mathbf{z}+ \mathbf{z}^{T}C^TC\mathbf{z}\big)\Big)
\\ & \ \ \ \ \ \ \ \ \ \ \ \ \ \ \ \ \ \ \ \ \ \times
e\Big(\alpha\big(2\mathbf{x}^{T}\gamma_1y
+ay^2+2\mathbf{z}^{T}\upsilon^{T}y\big)\Big)\Lambda(\mathbf{x})
\Lambda(\mathbf{y}) \Lambda(\mathbf{z}).\end{align*} We introduce
new variables $w\in \Z$ and $\mathbf{h}\in \Z^2$ to replace
$2\mathbf{x}^{T}\gamma_1 +ay+2\mathbf{z}^{T}\upsilon^{T}$ and
$C\mathbf{z}$, respectively. Therefore, we have
\begin{align*} S(\alpha)=& \int_{[0,1]^3}
\sum_{\substack{\mathbf{h}\in \Z^2
\\ |\mathbf{h}|\ll X}}\sum_{\substack{|w|\ll X}}\sum_{\substack{\mathbf{x}\in \N^3
\\ 1\le \mathbf{x}\le X}} \sum_{ 1\le y\le X} \sum_{\substack{\mathbf{z}\in
\N^{n-5}\\ 1\le \mathbf{z}\le X}}\Lambda(\mathbf{x})
\Lambda(\mathbf{y}) \Lambda(\mathbf{z})
\\ & \ \ \ \ \ \ \ \ \ \ \ \ \ \ \ \ \ \ \ \ \ \times
e\Big(\alpha\big(\mathbf{x}^{T}A_1\mathbf{x}+2\mathbf{x}^T(\gamma_1,\gamma_2)\mathbf{h}
+\mathbf{z}^{T}D\mathbf{z}+ \mathbf{h}^{T}H\mathbf{h}+wy\big)\Big)
\\ & \ \ \ \ \ \ \ \ \ \ \ \ \ \ \ \ \ \ \ \ \ \times
e\Big((2\mathbf{x}^{T}\gamma_1
+ay+2\mathbf{z}^{T}\upsilon^{T}-w)\beta_1\Big)
\\ & \ \ \ \ \ \ \ \ \ \ \ \ \ \ \ \ \ \ \ \ \ \times
e\Big((C\mathbf{z}-\mathbf{h})^T\boldsymbol{\beta}'\Big)d\boldsymbol{\beta},\end{align*}
where $\boldsymbol{\beta}=(\beta_1, \beta_2, \beta_3)^T$,
$\boldsymbol{\beta}'=(\beta_2, \beta_3)^T$ and
$d\boldsymbol{\beta}=d\beta_1d\beta_2d\beta_3$. Now we introduce
\begin{align*} \mathcal{F}(\alpha,\boldsymbol{\beta})=&
\sum_{\substack{\mathbf{h}\in \Z^2
\\ |\mathbf{h}|\ll X}}\sum_{\substack{\mathbf{x}\in \N^3
\\ 1\le \mathbf{x}\le X}}e\Big(\alpha\big(\mathbf{x}^{T}A_1\mathbf{x}+2\mathbf{x}^T(\gamma_1,\gamma_2)\mathbf{h}
+ \mathbf{h}^{T}H\mathbf{h}\big)\Big) \\ & \ \ \ \ \ \ \ \ \ \ \ \
\ \ \ \ \times e\Big(2\mathbf{x}^{T}\gamma_1
\beta_1-\mathbf{h}^T\boldsymbol{\beta}'\Big)\Lambda(\mathbf{x}),\end{align*}
and\begin{align*} \mathcal{H}(\alpha,\boldsymbol{\beta})=&
\sum_{|w|\ll X}\sum_{\substack{1\le y\le X}} e\Big(\alpha wy
+(ay-w)\beta_1\Big)\Lambda(y).\end{align*} On recalling notations
in (\ref{notationCv}), we define
\begin{align*} f_j(\alpha,\boldsymbol{\beta})=& \sum_{\substack{1\le z\le X}} e\Big(\alpha d_j
z^2
+2zv_jb_1+z\xi_j^{T}\boldsymbol{\beta}'\Big)\Lambda(z).\end{align*}
Then we obtain from above
\begin{align}\label{intcase3} \int_{\mathfrak{m}}\big|S(\alpha)\big|d\alpha\le& \int_{\mathfrak{m}}\int_{[0,1]^3}
\Big|\mathcal{F}(\alpha,\boldsymbol{\beta})\mathcal{H}(\alpha,\boldsymbol{\beta})
 \prod_{j=1}^{n-4}
f_j(\alpha,\boldsymbol{\beta})\Big|d\boldsymbol{\beta}d\alpha.\end{align}
By Lemmas \ref{lemma55} and \ref{lemma56},
\begin{align*} &\int_{[0,1]^4} \Big|\mathcal{H}(\alpha,\boldsymbol{\beta})
f_i(\alpha,\boldsymbol{\beta})f_j(\alpha,\boldsymbol{\beta})
f_k(\alpha,\boldsymbol{\beta})\Big|^2d\boldsymbol{\beta}d\alpha
\\ \ll & L^8 \sum_{\substack{ |w|,|w'|,|y|,|y'|,|z_i|,|z_i'|,
|z_j|,|z_j'|,|z_k|,|z_k'|\ll X \\
wy'+yw'+d_iz_iz_i'+d_jz_jz_j'+d_kz_kz_k'=0
\\ ay-w+v_iz_i+v_jz_j+v_kz_k=0
\\ z_i\xi_i+z_j\xi_j+z_k\xi_k=0}}1.\end{align*}
If $\rank(R_{i,j,k})=3$, then we can represent $z_i,z_j$ and $z_k$
by linear functions of $y$ and $w$.  Then by Lemma \ref{lemma57},
\begin{align*} &\int_{[0,1]^4} \Big|\mathcal{H}(\alpha,\boldsymbol{\beta})
f_i(\alpha,\boldsymbol{\beta})f_j(\alpha,\boldsymbol{\beta})
f_k(\alpha,\boldsymbol{\beta})\Big|^2d\boldsymbol{\beta}d\alpha
\ll X^5L^9.\end{align*} If $\rank(\xi_i,\xi_j)=2$, then we can
represent $z_i,z_j$ and $w$ by linear functions of $y$ and $z_k$.
Then we obtain by Lemma \ref{lemma57} again\begin{align*}
&\int_{[0,1]^4} \Big|\mathcal{H}(\alpha,\boldsymbol{\beta})
f_i(\alpha,\boldsymbol{\beta})f_j(\alpha,\boldsymbol{\beta})
f_k(\alpha,\boldsymbol{\beta})\Big|^2d\boldsymbol{\beta}d\alpha
\ll X^5L^9\end{align*}  provided that $d_k\not=0$. One can also
deduce from Lemmas \ref{lemma55}-\ref{lemma57} that
\begin{align*} &\int_{[0,1]^4}
|\mathcal{F}(\alpha,\boldsymbol{\beta})|^2d\boldsymbol{\beta}d\alpha
\ll X^5L^{7}.\end{align*}If $1\le i,j,k\le n-4$ are pairwise
distinct, then one has by (\ref{intcase1}) and the Cauchy-Schwarz
inequality
\begin{align*} \int_{\mathfrak{m}}|S(\alpha)|d\alpha \le &
\sup_{\substack{\alpha\in\mathfrak{m}\\ \boldsymbol{\beta}\in
[0,1]^3}}
\Big|\prod_{u\not=i,j,k}f_u(\alpha,\boldsymbol{\beta})\Big|\,
\Big(\int_{[0,1]^4}
\Big|\mathcal{F}(\alpha,\boldsymbol{\beta})\Big|^2d\boldsymbol{\beta}d\alpha\Big)^{1/2}
\\ & \times \Big(\int_{[0,1]^4}\Big |\mathcal{H}(\alpha,\boldsymbol{\beta})
f_i(\alpha,\boldsymbol{\beta})f_j(\alpha,\boldsymbol{\beta})
f_k(\alpha,\boldsymbol{\beta})\Big|^2d\boldsymbol{\beta}d\alpha\Big)^{1/2}.\end{align*}
Now it follows from above on together with Lemmas \ref{lemma43}
and \ref{lemma511}
\begin{align*} \int_{\mathfrak{m}}\big|S(\alpha)\big|d\alpha \ll X^{n-2}L^{-K/6}.\end{align*}
We complete the proof of Lemma \ref{lemma53}.

 \vskip3mm

\subsection{Proof of Lemma \ref{lemma54}}

Similar to Lemmas \ref{lemma58}-\ref{lemma510}, we also have the
following result.
\begin{lemma}\label{lemma512}If
$\rank(B_1)=\rank(B_2)=\rank(B_3)=4$, then we can write $A$ in the
form
\begin{align}\label{Afourth} A=\begin{pmatrix}A_1 & (\gamma_1,\gamma_2,\gamma_3)C
\\ C^{T}(\gamma_1,\gamma_2,\gamma_3)^{T} &  D+
C^{T}HC\end{pmatrix},\end{align} where $C\in M_{3,n-3}(\Z)$,
$\gamma_1,\gamma_2,\gamma_3\in \Q^{3}$, $H\in M_{3,3}(\Q)$ and
$D=\dia\{d_1,\ldots, d_{n-3}\}\in M_{n-3,n-3}(\Q)$ is a diagonal
matrix. Furthermore, we have $(\gamma_1,\gamma_2,\gamma_3)\in
GL_3(\Q)$.
\end{lemma}

\begin{lemma}\label{lemma513} Let $A$ be given by (\ref{Athird}) satisfying the
conditions in Lemma \ref{lemma510}. We write
\begin{align}\label{notationC}C=(\xi_1,\ldots,\xi_{n-4})\
.
\end{align}Then one can find pairwise distinct
$u_1,u_2,u_3,u_4,u_5,u_6$ with $1\le u_1,u_2,u_3,u_4,u_5,u_6\le
n-3$ so that $\rank(\xi_{u_1},\xi_{u_2},\xi_{u_3})=3$ and
$d_{u_4}d_{u_5}d_{u_6}\not=0$.
\end{lemma}
\begin{proof}
It follows from $\rank(A)\ge 9$ that $\rank(D)\ge 3$. If
$\rank(D)=3$, then we may assume that $d_1d_2d_3\not=0$ and
$d_j=0$ for $j\ge 4$. Thus $\rank(\xi_4,\ldots,\xi_{n-3})=3$, and
the desired conclusion follows. Next we assume $\rank(D)\ge 4$.
Since $\rank(\xi_1,\xi_2,\xi_3)=3$, the desired conclusion follows
again if there are distinct $k_1,k_2$ and $k_3$ such that
$d_{k_1}d_{k_2}d_{k_3}\not=0$ and $k_1,k_2,k_3\ge 4$. Thus we now
assume that for any distinct $k_1,k_2,k_3\ge 4$, one has
$d_{k_1}d_{k_2}d_{k_3}=0$. This yields $\rank(D)\le 5$. We first
consider the case $\rank(D)=4$. There are at least two distinct
$j_1,j_2\le 3$ such that $d_{j_1}d_{j_2}\not=0$. Suppose that
$d_{s_{i}}=0$ for $1\le i\le n-7$. Then the rank of
$\{\xi_{s_i}\}_{1\le i\le n-7}$ is at least 2, say
$\rank(\xi_{s_1},\xi_{s_2})=2$. Since $\xi_{j_1}$ and $\xi_{j_2}$
are linear independent for $1\le j_1\not=j_2\le 3$, one has either
$\rank(\xi_{j_2},\xi_{s_1},\xi_{s_2})=3$ or
$\rank(\xi_{j_1},\xi_{s_1},\xi_{s_2})=3$. The desired conclusion
follows easily by choosing $u_1=j_1$ or $j_2$, $u_2=s_1$ and
$u_3=s_2$. Now we consider the case $\rank(D)=5$, and we may
assume that $d_1d_2d_3d_4d_5\not=0$ and $d_r=0$ for $r\ge 6$.
Since $\rank(A)\ge 9$, there exist $r\ge 6$ (say $r=6$) such that
$\xi_r\not=0$. Then one can choose $j_1,j_2\le 3$ so that
$\rank(\xi_{j_1},\xi_{j_2},\xi_{6})=3$. The desired conclusion
follows by choosing $u_1=j_1$, $u_2=j_2$ nd $u_3=6$. The proof of
Lemma \ref{lemma513} is completed.\end{proof}

\noindent{\textit{Proof of Lemma \ref{lemma54}}}. We apply Lemma
\ref{lemma512} to conclude that
\begin{align*} S(\alpha)=& \sum_{\substack{\mathbf{x}\in \N^3
\\ 1\le \mathbf{x}\le X}} \sum_{\substack{\mathbf{y}\in
\N^{n-3}\\ 1\le \mathbf{y}\le X}}\Lambda(\mathbf{x})
\Lambda(\mathbf{y})e\Big(\alpha(\mathbf{x}^{T}A_1\mathbf{x}
+\mathbf{y}^{T}D\mathbf{y})\Big)
\\ & \ \ \ \ \ \ \ \ \ \ \ \ \ \ \ \ \ \times
e\Big(\alpha\big(2\mathbf{x}^T(\gamma_1,\gamma_2,\gamma_3)C\mathbf{y}
+\mathbf{y}^{T}C^{T}HC\mathbf{y}\big)\Big).\end{align*} By
orthogonality, one has
\begin{align*} S(\alpha)=& \int_{[0,1]^3}\sum_{\substack{\mathbf{x}\in \N^3
\\ 1\le \mathbf{x}\le X}} \sum_{\substack{\mathbf{y}\in
\N^{n-3}\\ 1\le \mathbf{y}\le X}} \sum_{\substack{\mathbf{z}\in
\Z^3
\\ |\mathbf{z}|\ll X}}\Lambda(\mathbf{x})
\Lambda(\mathbf{y})e\Big(\alpha\big(\mathbf{x}^{T}A_1\mathbf{x}
+\mathbf{y}^{T}D\mathbf{y}\big)\Big) \\ & \ \ \ \ \ \ \ \ \ \ \ \
\ \ \ \ \ \times
e\Big(\alpha(2\mathbf{x}^T(\gamma_1,\gamma_2,\gamma_3)\mathbf{z}
+\mathbf{z}^{T}H\mathbf{z})\Big) \\ & \ \ \ \ \ \ \ \ \ \ \ \ \ \
\ \ \ \ \ \ \ \times
e\Big((\mathbf{y}^{T}C^{T}-\mathbf{z}^{T})\boldsymbol{\beta}\Big)d\boldsymbol{\beta},\end{align*}
where $\boldsymbol{\beta}=(\beta_1,\beta_2,\beta_3)^T$ and
$d\boldsymbol{\beta}=d\beta_1d\beta_2d\beta_3$.
 Now we introduce
\begin{align*} \mathcal{F}(\alpha,\boldsymbol{\beta})=&\sum_{\substack{\mathbf{x}\in \N^3
\\ 1\le \mathbf{x}\le X}} \sum_{\substack{\mathbf{z}\in
\Z^3
\\ 1\le \mathbf{z}\le X}}\Lambda(\mathbf{x})
e\Big(\alpha\big(\mathbf{x}^{T}A_1\mathbf{x}
+2\mathbf{x}^T(\gamma_1,\gamma_2,\gamma_3)\mathbf{z}
+\mathbf{z}^{T}H\mathbf{z}\big)-\mathbf{z}^{T}\boldsymbol{\beta}\Big),\end{align*}
and
\begin{align*} f_j(\alpha,\boldsymbol{\beta})=& \sum_{\substack{1\le y\le X}} e\Big(d_j\alpha
y^2 +2y\xi_j^{T}\boldsymbol{\beta}\Big)\Lambda(y),\end{align*}
where $\xi_1,\ldots,\xi_{n-3}$ is given by (\ref{notationC}). We
conclude from above
\begin{align}\label{intcase4} \int_{\mathfrak{m}}\big|S(\alpha)\big|d\alpha\le& \int_{\mathfrak{m}}\int_{[0,1]^3}
\Big|\mathcal{F}(\alpha,\boldsymbol{\beta})\prod_{j=1}^{n-3}
f_j(\alpha,\boldsymbol{\beta})\Big|d\boldsymbol{\beta}d\alpha.\end{align}
One can easily deduce from Lemmas \ref{lemma55}-\ref{lemma57}
\begin{align}\label{eq527} &\int_{[0,1]^4} \Big|\prod_{i=1}^{5}
f_{u_i}(\alpha,\boldsymbol{\beta})\Big|^2d\boldsymbol{\beta}d\alpha
\ll X^5L^{11}\end{align}provided that
$\rank(\xi_{u_1},\xi_{u_2},\xi_{u_3})=3$ and
$d_{u_4}d_{u_5}\not=0$. Similarly, we also have
\begin{align}\label{eq528} \int_{[0,1]^4}
\Big|\mathcal{F}(\alpha,\boldsymbol{\beta})\Big|^2d\boldsymbol{\beta}d\alpha\ll
X^7L^7.\end{align}
 By (\ref{intcase4})
and the Cauchy-Schwarz inequality, one has for distinct
$u_1,u_2,u_3,u_4$ and $u_5$  that
\begin{align} \int_{\mathfrak{m}}\big|S(\alpha)\big|d\alpha \le &
\sup_{\substack{\alpha\in\mathfrak{m}\\ \boldsymbol{\beta}\in
[0,1]^3}}
\Big|\prod_{k\not=u_1,u_2,u_3,u_4,u_5}f_k(\alpha,\boldsymbol{\beta})\Big|
\Big(\int_{[0,1]^4}
\Big|\mathcal{F}(\alpha,\boldsymbol{\beta})\Big|^2d\boldsymbol{\beta}d\alpha\Big)^{1/2}\notag
\\ & \times \Big(\int_{[0,1]^4} |\prod_{i=1}^{5}
f_{u_i}(\alpha,\boldsymbol{\beta})\Big|^2d\boldsymbol{\beta}d\alpha\Big)^{1/2}.\label{eq529}\end{align}
Combining (\ref{eq527})-(\ref{eq529}) and  Lemmas
\ref{lemma43}-\ref{lemma513}, one has
\begin{align*} \int_{\mathfrak{m}}\big|S(\alpha)\big|d\alpha \ll X^{n-2}L^{-K/6}.\end{align*}
The proof of Lemma \ref{lemma54} is finished.

\vskip3mm
\section{Quadratic forms with off-diagonal rank $\ge 4$}

\begin{proposition}\label{proposition61}
Let $A$ be defined in (\ref{defA}), and let $S(\alpha)$ be defined
in (\ref{defSalpha}). We write
\begin{align}\label{defG} G=\begin{pmatrix}a_{1,5} & \cdots &
a_{1,9}
\\ \vdots & \cdots & \vdots
\\ a_{5,5} & \cdots & a_{5,9}\end{pmatrix}.\end{align}Suppose
that $\det(G)\not=0$. Then we have
\begin{align*}\int_{\mathfrak{m}}\big|S(\alpha)\big| d\alpha \ll
X^{n-2}L^{-K/20},\end{align*} where the implied constant depends
on $A$ and $K$.
\end{proposition}
\begin{remark}In view of the condition on $G$, one has $n\ge 9$. Proposition \ref{proposition61}
 implies the asymptotic formula (\ref{asymptotic}) for a wide class of quadratic forms with $n\ge 9$.
 However, different from Proposition \ref{proposition51}, it does not require the condition $\rank(A)\ge 9$
  in Proposition \ref{proposition61}. Therefore, Proposition \ref{proposition61} also
  implies
the asymptotic formula (\ref{asymptotic}) for some quadratic
forms, which can not be covered by Theorem \ref{theorem1}.
  \end{remark} Throughout this section, we shall assume that the matrix $G$
given by (\ref{defG}) is invertible.

\begin{lemma}\label{lemma61}
Let $\tau\not=0$ be a real number. Then we have
\begin{align*}\int_{\mathfrak{m}(Q)}\int_{\mathfrak{m}(Q)}\sum_{|x|\ll X}
\min\{X,\ \|x\tau(\alpha-\beta)\|^{-1}\}d\alpha d\beta \ll
LQ^{7/2}X^{-2},\end{align*}where the implied constant depends on
$\tau$.\end{lemma}
\begin{proof} Without loss of generality, we assume that $0<|\tau|\le
1$. Thus $|\tau(\alpha-\beta)|\le 1$. We introduce
\begin{align*}\mathcal{M}:=
\mathcal{M}=\bigcup_{1\le q\le Q^{1/2}}\bigcup_{\substack{-q\le
a\le q
\\ (a,q)=1}}\Big\{\big|\alpha-\frac{a}{q}\big|\le \frac{Q^{1/2}}{qX^2}\Big\}\end{align*}
Dirichlet's approximation theorem, there exist $a$ and $q$ with
$1\le a\le q\le X^2Q^{-1/2}$, $(a,q)=1$ and $|\alpha-a/q|\le
Q^{1/2}(qX^2)^{-1}$. Since $|\tau(\alpha-\beta)|\le 1$, one has
$-q\le a\le q$. If $\tau(\alpha-\beta)\not\in \mathcal{M}$, then
$q>Q^{1/2}$. By Lemma 2.2 of Vaughan \cite{V},
\begin{align*}\sum_{|x|\ll X}\min\{X,\
\|x\tau(\alpha-\beta)\|^{-1}\}\ll LQ^{-1/2}X^{2}.\end{align*}
Therefore, we obtain
\begin{align*}&\int_{\mathfrak{m}(Q)}\int_{\substack{\mathfrak{m}(Q)
\\ \tau(\alpha-\beta)\not\in \mathcal{M}}}
\sum_{|x|\ll X}\min\{X,\ \|x\tau(\alpha-\beta)\|^{-1}\}d\alpha
d\beta \\ \ll&\
LQ^{-1/2}X^{2}\int_{\mathfrak{m}(Q)}\int_{\mathfrak{m}(Q)}d\alpha
d\beta \ \ll \ LQ^{7/2}X^{-2}.\end{align*} When
$\tau(\alpha-\beta)\in \mathcal{M}$, we apply the trivial bound to
the summation over $x$ to deduce
that\begin{align*}&\int_{\mathfrak{m}(Q)}\int_{\substack{\mathfrak{m}(Q)
\\ \tau(\alpha-\beta)\in \mathcal{M}}}
\sum_{|x|\ll X}\min\{X,\ \|x\tau(\alpha-\beta)\|^{-1}\}d\alpha
d\beta \\ \ll&\
X^{2}\int_{\mathfrak{m}(Q)}\int_{\substack{\mathfrak{m}(Q)
\\ \tau(\alpha-\beta)\in \mathcal{M}}}d\alpha d\beta
\ \ll \ X^2(Q^2X^{-2}QX^{-2})=Q^{3}X^{-2}.\end{align*} The desired
conclusion follows from above immediately.
\end{proof}
To introduce the next lemma, we define
\begin{align*}\Phi(\alpha)=\min\{X,\ \|\alpha\|^{-1}\}.\end{align*}
For $\mathbf{v}=(v_1,\ldots,v_5)\in \Z^5$ and $G$ given by
(\ref{defG}), we write
\begin{align}2G\mathbf{v}=\begin{pmatrix}g_1(\mathbf{v})
\\ \vdots \\ g_5(\mathbf{v})\end{pmatrix}.\end{align}
\begin{lemma}\label{lemma62}
One
has\begin{align}\label{means2}\int_{\mathfrak{m}(Q)}\big|S(\alpha)\big|^2d\alpha\ll
X^{2n-10}L^{2n-6}\int_0^1\Big(\int_{\mathfrak{m}(Q)}J_{\gamma}(\alpha)d\alpha\Big)\Phi(\gamma)d\gamma,\end{align}
where\begin{align}\label{defjgamma}J_{\gamma}(\alpha)=\sum_{|\mathbf{v}|\le
X}\Big|\sum_{|z|\le X}\Lambda(z)\Lambda(z+v_1)e(\alpha z
g_5(\mathbf{v}))e(\gamma
z)\Big|\prod_{j=1}^4\Phi\big(g_j(\mathbf{v})\alpha\big).\end{align}
\end{lemma}
\begin{proof}Let
\begin{align*}r(\mathbf{y})=\sum_{i=5}^9\sum_{j=1}^4a_{i,j}y_iy_j, \
q(\mathbf{z})=\sum_{i=5}^9\sum_{j=5}^9a_{i,j}z_iz_j \ \textrm{ and
}\
p(\mathbf{w})=\sum_{i=9}^n\sum_{j=9}^na_{i,j}w_iw_j.\end{align*}We
set
\begin{align*}B=(2a_{i,j})_{\substack{1\le i\le 4 ,10\le j\le n}} \ \textrm{ and }\
C=(2a_{i,j})_{5\le i\le 9,10\le j\le n}.\end{align*} Then $f$ can
be written in the form
\begin{align*}f(\mathbf{x})=r(\mathbf{y})+y_1g_1(\mathbf{z})+\cdots+y_sg_s(\mathbf{z})+q(\mathbf{z})+
\mathbf{y}^TB\mathbf{w}+\mathbf{z}^TC\mathbf{w}+p(\mathbf{w}),\end{align*}
where $\mathbf{z}=(z_1,\ldots,z_5)$,
$\mathbf{y}=(y_1,\ldots,y_4)$, $\mathbf{w}=(w_1,\ldots,w_{n-9})$.
Note that
$\mathbf{y}^TB\mathbf{w}+\mathbf{z}^TC\mathbf{w}+p(\mathbf{w})$
vanishes if $n=9$. Therefore, one has
\begin{align*}S(\alpha)=\sum_{\substack{1\le \mathbf{y}\le X
\\ 1\le \mathbf{w}\le X }} &\sum_{1\le \mathbf{z}\le
X}\Lambda(\mathbf{z})
e\Big(\alpha\big(y_1g_1(\mathbf{z})+\cdots+y_sg_s(\mathbf{z})+q(\mathbf{z})
+ \mathbf{z}^TB\mathbf{w}\big)\Big)
\\ & \times\Lambda(\mathbf{y})\Lambda(\mathbf{w})e\Big(\alpha\big(r(\mathbf{y})
+ \mathbf{y}^TB\mathbf{w}+p(\mathbf{w})\big)\Big).\end{align*} By
Cauchy's inequality,
\begin{align}\label{ST}|S(\alpha)|^2\le & X^{n-5}L^{2n-10}T(\alpha),
\end{align}where
\begin{align*}T(\alpha)=\sum_{\substack{1\le \mathbf{y}\le X
\\ 1\le \mathbf{w}\le X }}\Big|
\sum_{1\le \mathbf{z}\le X}\Lambda(\mathbf{z})
e\Big(\alpha\big(\sum_{j=1}^4y_jg_j(\mathbf{z})+q(\mathbf{z})+\mathbf{z}^TC\mathbf{w}\big)\Big)\Big|^2.
\end{align*}
Then we deduce that
\begin{align*}T(\alpha)= &  \sum_{\substack{1\le \mathbf{y}\le X
\\ 1\le \mathbf{w}\le X }}
\sum_{1\le \mathbf{z}_1\le X}\sum_{1\le \mathbf{z}_2\le
X}\Lambda(\mathbf{z}_1)\Lambda(\mathbf{z}_2)
e\Big(\alpha\big(\sum_{j=1}^sy_jg_j(\mathbf{z}_1-\mathbf{z}_2)
+q(\mathbf{z}_1)-q(\mathbf{z}_2)\big)\Big)
\\ &\ \ \ \ \ \ \ \ \ \ \ \ \ \ \ \ \ \ \ \ \ \ \ \ \ \ \ \ \
\times e\Big(\alpha(\mathbf{z}_1-\mathbf{z}_2)^TC\mathbf{w}\Big)
\\ = &
\sum_{1\le \mathbf{z}_1\le X}\sum_{1\le \mathbf{z}_2\le
X}\Lambda(\mathbf{z}_1)\Lambda(\mathbf{z}_2) \sum_{\substack{1\le
\mathbf{y}\le X
\\ 1\le \mathbf{w}\le X }}e\Big(\alpha\big(\sum_{j=1}^sy_jg_j(\mathbf{z}_1-\mathbf{z}_2)
+q(\mathbf{z}_1)-q(\mathbf{z}_2)\big)\Big)
\\ &\ \ \ \ \ \ \ \ \ \ \ \ \ \ \ \ \ \ \ \ \ \ \ \ \ \ \ \ \
\times e\Big(\alpha(\mathbf{z}_1-\mathbf{z}_2)^TC\mathbf{w}\Big).
\end{align*}
By changing variables $\mathbf{z}_2=\mathbf{z}_1+\mathbf{v}$, we
have
\begin{align*}T(\alpha)=&
\sum_{1\le \mathbf{z}\le X}\sum_{\substack{|\mathbf{v}|\le X
\\ 1\le \mathbf{v}+\mathbf{z}\le X}}\Lambda(\mathbf{z})\Lambda(\mathbf{z}+\mathbf{v})
\sum_{\substack{1\le \mathbf{y}\le X
\\ 1\le \mathbf{w}\le X }}e\Big(\alpha\big(\sum_{j=1}^sy_jg_j(\mathbf{v})
+q(\mathbf{z}+\mathbf{v})-q(\mathbf{z})\big)\Big)
\\ &\ \ \ \ \ \ \ \ \ \ \ \ \ \ \ \ \ \ \ \ \ \ \ \ \ \ \ \ \
\times e\Big(\alpha(\mathbf{z}_1-\mathbf{z}_2)^TC\mathbf{w}\Big).
\end{align*}
We exchange the summation over $\mathbf{z}$ and the summation over
$\mathbf{v}$ to obtain
\begin{align}\label{newT}T(\alpha)= &
\sum_{|\mathbf{v}|\le X}\Big( \sum_{1\le \mathbf{y}\le
X}e\big(\alpha\sum_{j=1}^4y_jg_j(\mathbf{v})\big)\Big)
R(\mathbf{v}) \prod_{j=1}^5\mathcal{K}_{j,\mathbf{v}}(\alpha),
\end{align}
where
\begin{align*}R(\mathbf{v}) =e\Big(\alpha q(\mathbf{v})\Big)
\sum_{\substack{1\le \mathbf{w}\le X
}}e\Big(\alpha(\mathbf{v}^TC\mathbf{w})\Big)
\end{align*}and
\begin{align}\label{Kj}\mathcal{K}_{j,\mathbf{v}}(\alpha)=\sum_{\substack{1\le z_j \le X
\\ 1-v_j\le z_j\le X-v_j}}\Lambda(z_j)\Lambda(z_j+v_j)
e\big(2\alpha\sum_{k=1}^5a_{j+4,k+4}v_k\big).
\end{align}The range of $z_j$ in summation (\ref{Kj}) depends on
$v_j$. We first follow the standard argument (see for example the
argument around (15) in \cite{W3}) to remove the dependence on
$v_j$. We write
\begin{align}\label{Ggamma}\mathcal{G}_{v_1}(\gamma)=\sum_{\substack{1\le z \le X
\\ 1-v_1\le z \le X-v_1}}
e\big(-z\gamma\big)
\end{align}
and\begin{align}\label{K0}\mathcal{K}_{0,\mathbf{v}}(\alpha,\gamma)=\sum_{\substack{|z|
\le X}}\Lambda(z)\Lambda(z+v_1) e\big(\alpha
g_5(\mathbf{v})\big)e(\gamma z).
\end{align}
Then we deduce from (\ref{Kj}), (\ref{Ggamma}) and (\ref{K0}) that
\begin{align}\label{Kint}\mathcal{K}_{1,\mathbf{v}}(\alpha)=\int_{0}^{1}
\mathcal{K}_{0,\mathbf{v}}(\alpha,\gamma)\mathcal{G}_{v_1}(\gamma)d\gamma.
\end{align}
On substituting (\ref{Kint}) into (\ref{newT}), we obtain
\begin{align*}T(\alpha)= &
\sum_{|\mathbf{v}|\le X}R(\mathbf{v})
\prod_{j=2}^5\mathcal{K}_{j,\mathbf{v}}(\alpha)\Big( \sum_{1\le
\mathbf{y}\le
X}e\big(\alpha\sum_{j=1}^4y_jg_j(\mathbf{v})\big)\Big)\int_{0}^{1}
\mathcal{K}_{0,\mathbf{v}}(\alpha,\gamma)\mathcal{G}_{v_1}(\gamma)d\gamma
\\ =& \int_{0}^{1}
\sum_{|\mathbf{v}|\le X}R(\mathbf{v})
\prod_{j=2}^5\mathcal{K}_{j,\mathbf{v}}(\alpha)\Big( \sum_{1\le
\mathbf{y}\le
X}e\big(\alpha\sum_{j=1}^4y_jg_j(\mathbf{v})\big)\Big)
\mathcal{K}_{0,\mathbf{v}}(\alpha,\gamma)\mathcal{G}_{v_1}(\gamma)d\gamma.
\end{align*}
Then we conclude that
\begin{align}\label{FinalT}|T(\alpha)|\ll  X^{n-5}L^{4}\int_{0}^{1}
\sum_{ |\mathbf{v}|\le X}\big|
\mathcal{K}_{0,\mathbf{v}}(\alpha,\gamma)\big|\prod_{j=1}^4\Phi\big(g_j(\mathbf{v}\alpha)\big)\Phi(\gamma)d\gamma,
\end{align}By putting (\ref{FinalT}) into (\ref{ST}), one has
\begin{align*}|S(\alpha)|^2\ll  X^{2n-10}L^{2n-6}\int_{0}^{1}
\sum_{ |\mathbf{v}|\le X}\big|
\mathcal{K}_{0,\mathbf{v}}(\alpha,\gamma)\big|\prod_{j=1}^4\Phi\big(g_j(\mathbf{v}\alpha)\big)\Phi(\gamma)d\gamma.
\end{align*}
Therefore,
\begin{align*}\int_{\mathfrak{m}(Q)}\big|S(\alpha)\big|^2d\alpha\ll
X^{2n-10}L^{2n-6}\int_0^1\Big(\int_{\mathfrak{m}(Q)}J_{\gamma}(\alpha)d\alpha\Big)\Phi(\gamma)d\gamma.\end{align*}
The proof is completed.
\end{proof}
\begin{lemma}\label{lemma63} Let $J_{\gamma}(\alpha)$ be defined in (\ref{defjgamma}).
 Then one has uniformly for $\gamma\in [0,1]$ that \begin{align*}\int_{\mathfrak{m}(Q)}J_{\gamma}(\alpha)d\alpha
\ll L^{25/4}Q^{-17/8}X^8.
\end{align*}\end{lemma}
\begin{proof}We deduce by changing variables $\mathbf{h}=2G\mathbf{v}$ that
\begin{align*}J_{\gamma}(\alpha)=\sum_{\substack{|\mathbf{h}|\le
cX\\ (2G)^{-1}\mathbf{h}\in \Z^5\\
|(2G)^{-1}\mathbf{h}|\le X}}\Big|\sum_{|z|\le
X}\Lambda(z)\Lambda(z+\sum_{j=1}^5b_jh_j)e(\alpha z h_5)e(\gamma
z)\Big|\prod_{j=1}^4\Phi(h_j\alpha)\end{align*} for some constants
$c,b_1,\ldots,b_5$ depending only on $G$. We point out that
$b_1,\ldots, b_5$ are rational numbers, and we extend the domain
of function $\Lambda(x)$ by taking $\Lambda(x)=0$ if
$x\in\Q\setminus\N$. Then we have
\begin{align*}J_{\gamma}(\alpha)\le \sum_{\substack{|\mathbf{u}|\le
cX}}\sum_{|h|\le cX}\Big|\sum_{|z|\le
X}\Lambda(z)\Lambda(z+\sum_{j=1}^4b_ju_j+b_5h)e(\alpha z
h)e(\gamma z)\Big|\prod_{j=1}^4\Phi(u_j\alpha).\end{align*} We
first handle the easier case $b_5=0$. In this case, we can easily
obtain a nontrivial estimate for the summation over $h$. By
Cauchy's inequality and Lemma \ref{lemma41}, one has
\begin{align*}&\Big(\sum_{|h|\le cX}\Big|\sum_{|z|\le
X}\Lambda(z)\Lambda(z+\sum_{j=1}^4b_ju_j)e(\alpha z h)e(\gamma
z)\Big|\Big)^2 \\ \le & (2cX+1)\sum_{|h|\le cX}\Big|\sum_{|z|\le
X}\Lambda(z)\Lambda(z+\sum_{j=1}^4b_ju_j)e(\alpha z h)e(\gamma
z)\Big|^2 \\ \ll & X^2\sum_{|x|\ll X}\min\{X,\
\|x\alpha\|^{-1}\}.\end{align*} For $\alpha\in\mathfrak{m}(Q)$, we
apply Lemma \ref{lemma42} to deduce from above
\begin{align*}\sum_{|h|\le cX}\Big|\sum_{|z|\le
X}\Lambda(z)\Lambda(z+\sum_{j=1}^4b_ju_j)e(\alpha z h)e(\gamma
z)\Big|\ll L^{1/2}Q^{-1/2}X^2.\end{align*} Then for
$\alpha\in\mathfrak{m}(Q)$, we obtain
\begin{align*}J_{\gamma}(\alpha)\ll L^{1/2}Q^{-1/2}X^2\sum_{\substack{|\mathbf{u}|\le
cX}}\prod_{j=1}^4\Phi(u_j\alpha)\ll
L^{9/2}Q^{-9/2}X^{10},\end{align*} and thereby
\begin{align}\label{zero}\int_{\mathfrak{m}(Q)}J_{\gamma}(\alpha)
d\alpha \ll L^{9/2}Q^{-5/2}X^{8}\end{align} provided that $b_5=0$.
From now on, we assume $b_5\not=0$. Then we have
\begin{align*}&\sum_{|h|\le cX}\Big|\sum_{|z|\le
X}\Lambda(z)\Lambda(z+\sum_{j=1}^4b_ju_j+b_5h)e(\alpha z
h)e(\gamma z) \Big|\\ =& \sum_{\substack{|k|\le c'X\\
\frac{1}{b_5}(k-\sum_{j=1}^4b_ju_j)\in \Z \\
|\frac{1}{b_5}(k-\sum_{j=1}^4b_ju_j)|\le cX}}\Big|\sum_{|z|\le
X}\Lambda(z)\Lambda(z+k)e\big(\frac{\alpha}{b_5}
z(k-\sum_{j=1}^4b_ju_j)\big)e(\gamma z) \Big|\end{align*} for some
constant $c'$ depending only on $b_1,\ldots,b_5$ and $c$.
Therefore, one has
\begin{align*}&\sum_{|h|\le cX}\Big|\sum_{|z|\le
X}\Lambda(z)\Lambda(z+\sum_{j=1}^4b_ju_j+b_5h)e(\alpha z
h)e(\gamma z) \Big|\\ \le & \sum_{\substack{|k|\le
c'X}}\Big|\sum_{|z|\le
X}\Lambda(z)\Lambda(z+k)e\big(\frac{\alpha}{b_5}
z(k-\sum_{j=1}^4b_ju_j)\big)e(\gamma z) \Big|.\end{align*} We
apply Cauchy's inequality to deduce that
\begin{align*}&\sum_{|h|\le cX}\Big|\sum_{|z|\le
X}\Lambda(z)\Lambda(z+\sum_{j=1}^4b_ju_j+b_5h)e(\alpha z
h)e(\gamma z) \Big|\\ \le &
(2c'X+1)^{1/2}\Big(\sum_{\substack{|k|\le c'X}}\Big|\sum_{|z|\le
X}\Lambda(z)\Lambda(z+k)e\big(\frac{\alpha}{b_5}
z(k-\sum_{j=1}^4b_ju_j)\big)e(\gamma z)
\Big|^2\Big)^{1/2}.\end{align*} We apply Cauchy's inequality again
to obtain\begin{align*}J_{\gamma}(\alpha)\le
(2c'X+1)^{1/2}\Xi_{\gamma}(\alpha)^{1/2}\Big(\sum_{\substack{|\mathbf{u}|\le
cX}}\prod_{j=1}^4\Phi(u_j\alpha)\Big)^{1/2},\end{align*} where
$\Xi_{\gamma}(\alpha)$ is defined as
\begin{align*}\Xi_{\gamma}(\alpha)=\sum_{\substack{|\mathbf{u}|\le
cX}}\sum_{\substack{|k|\le c'X}}\Big|\sum_{|z|\le
X}\Lambda(z)\Lambda(z+k)e\big(\frac{\alpha}{b_5}
z(k-\sum_{j=1}^4b_ju_j)\big)e(\gamma z)
\Big|^2\prod_{j=1}^4\Phi(u_j\alpha).\end{align*} By Lemma
\ref{lemma32},\begin{align*}J_{\gamma}(\alpha)\ll
L^{2}Q^{-2}X^{9/2} \Xi_{\gamma}(\alpha)^{1/2}.\end{align*}
Therefore, we have
\begin{align}\label{intJ}\int_{\mathfrak{m}(Q)}J_{\gamma}(\alpha)
d\alpha \ll &L^{2}Q^{-2}X^{9/2}\Big(\int_{\mathfrak{m}(Q)} d\alpha
\Big)^{1/2}\Big(\int_{\mathfrak{m}(Q)}\Xi_{\gamma}(\alpha) d\alpha
\Big)^{1/2} \notag
\\ \ll & L^{2}Q^{-1}X^{7/2}\Big(\int_{\mathfrak{m}(Q)}\Xi_{\gamma}(\alpha) d\alpha
\Big)^{1/2}.\end{align} Now it suffices to estimate
$\int_{\mathfrak{m}(Q)}\Xi_{\gamma}(\alpha) d\alpha$. We observe
\begin{align*}&\int_{\mathfrak{m}(Q)}\Xi_{\gamma}(\alpha)d\alpha
\\ =&\int_{\mathfrak{m}(Q)}\sum_{\substack{|\mathbf{u}|\le cX}}\sum_{\substack{|k|\le
c'X}}\sum_{|z_1|\le X}\sum_{|z_2|\le X}
\varpi(z_1,z_2,k)e\big(\frac{\alpha}{b_5}
(z_1-z_2)k\big)\Pi(\alpha,\mathbf{u},z_1,z_2)d\alpha
\\ =&\int_{\mathfrak{m}(Q)}\sum_{\substack{|k|\le
c'X}}\sum_{|z_1|\le X}\sum_{|z_2|\le X}
\varpi(z_1,z_2,k)e\big(\frac{\alpha}{b_5}
(z_1-z_2)k\big)\sum_{\substack{|\mathbf{u}|\le
cX}}\Pi(\alpha,\mathbf{u},z_1,z_2)d\alpha,\end{align*} where
\begin{align*}\varpi(z_1,z_2,k)=\Lambda(z_1)\Lambda(z_1+k)\Lambda(z_2)\Lambda(z_2+k)e(\gamma (z_1-z_2))
\end{align*}
and
\begin{align*}\Pi(\alpha,\mathbf{u},z_1,z_2)=e\big(\frac{\alpha}{b_5}
(z_1-z_2)\sum_{j=1}^4b_ju_j)\big)\prod_{j=1}^4\Phi(u_j\alpha).\end{align*}
We exchange the order of summation and integration to conclude
that
\begin{align*}&\int_{\mathfrak{m}(Q)}\Xi_{\gamma}(\alpha)d\alpha
\\ =&\sum_{\substack{|k|\le
c'X}}\sum_{|z_1|\le X}\sum_{|z_2|\le X}
\varpi(z_1,z_2,k)\int_{\mathfrak{m}(Q)}e\big(\frac{\alpha}{b_5}
(z_1-z_2)k\big)\sum_{\substack{|\mathbf{u}|\le
cX}}\Pi(\alpha,\mathbf{u},z_1,z_2)d\alpha \\ \ll
&L^4\sum_{\substack{|k|\le c'X}}\sum_{|z_1|\le X}\sum_{|z_2|\le X}
\Big|\int_{\mathfrak{m}(Q)}e\big(\frac{\alpha}{b_5}
(z_1-z_2)k\big)\sum_{\substack{|\mathbf{u}|\le
cX}}\Pi(\alpha,\mathbf{u},z_1,z_2)d\alpha\Big|\\ =
&L^4\sum_{|z_1|\le X}\sum_{|z_2|\le X}\sum_{\substack{|k|\le c'X}}
\Big|\int_{\mathfrak{m}(Q)}e\big(\frac{\alpha}{b_5}
(z_1-z_2)k\big)\sum_{\substack{|\mathbf{u}|\le
cX}}\Pi(\alpha,\mathbf{u},z_1,z_2)d\alpha\Big|.\end{align*} Then
the Cauchy-Schwarz inequality implies
\begin{align}&\Big(\int_{\mathfrak{m}(Q)}\Xi_{\gamma}(\alpha)d\alpha\Big)^2\notag
\\ \ll &L^8X^3\sum_{|z_1|\le X}\sum_{|z_2|\le X}\sum_{\substack{|k|\le
c'X}} \Big|\int_{\mathfrak{m}(Q)}e\big(\frac{\alpha}{b_5}
(z_1-z_2)k\big)\sum_{\substack{|\mathbf{u}|\le
cX}}\Pi(\alpha,\mathbf{u},z_1,z_2)d\alpha\Big|^2.\label{meanxi2}\end{align}
Now we apply the method developed by the author \cite{Zhao} to
deduce that
\begin{align*}&\sum_{|z_1|\le X}\sum_{|z_2|\le X}\sum_{\substack{|k|\le
c'X}} \Big|\int_{\mathfrak{m}(Q)}e\big(\frac{\alpha}{b_5}
(z_1-z_2)k\big)\sum_{\substack{|\mathbf{u}|\le
cX}}\Pi(\alpha,\mathbf{u},z_1,z_2)d\alpha\Big|^2
\\=&\int_{\mathfrak{m}(Q)}\int_{\mathfrak{m}(Q)}\sum_{|z_1|\le X}\sum_{|z_2|\le X}\sum_{\substack{|k|\le
c'X}}e\big(\frac{\alpha-\beta}{b_5} (z_1-z_2)k\big)
\\ &\ \ \ \ \ \ \ \ \ \ \ \ \ \ \ \ \ \ \ \ \ \ \ \ \times\sum_{\substack{|\mathbf{u}_1|\le
cX}}\Pi(\alpha,\mathbf{u}_1,z_1,z_2)\sum_{\substack{|\mathbf{u}_2|\le
cX}}\Pi(-\beta,\mathbf{u}_2,z_1,z_2)d\alpha d\beta
\\ \le & \int_{\mathfrak{m}(Q)}\int_{\mathfrak{m}(Q)}\sum_{|z_1|\le X}\sum_{|z_2|\le X}\Big|\sum_{\substack{|k|\le
c'X}}e\big(\frac{\alpha-\beta}{b_5} (z_1-z_2)k\big)\Big|
\\ &\ \ \ \ \ \ \ \ \ \ \ \ \ \ \ \ \ \ \ \ \ \ \ \ \times\sum_{\substack{|\mathbf{u}_1|\le
cX}}\prod_{j=1}^4\Phi(u_j\alpha)\sum_{\substack{|\mathbf{u}_2|\le
cX}}\prod_{j=1}^4\Phi(u_j'\beta)d\alpha d\beta,\end{align*}
 where $\mathbf{u}_1=(u_1,\ldots,u_4)^T\in \Z^4$ and $\mathbf{u}_2=(u_1',\ldots,u_4')^T\in \Z^4$.
 Therefore, we obtain by Lemma \ref{lemma42}
\begin{align*}&\sum_{|z_1|\le X}\sum_{|z_2|\le X}\sum_{\substack{|k|\le
c'X}} \Big|\int_{\mathfrak{m}(Q)}e\big(\frac{\alpha}{b_5}
(z_1-z_2)k\big)\sum_{\substack{|\mathbf{u}|\le
cX}}\Pi(\alpha,\mathbf{u},z_1,z_2)d\alpha\Big|^2
\\ \ll &\int_{\mathfrak{m}(Q)}\int_{\mathfrak{m}(Q)}\sum_{|z_1|\le X}\sum_{|z_2|\le X}\min\{X,\
\|\frac{\alpha-\beta}{b_5}
(z_1-z_2)\|^{-1}\}(L^{4}Q^{-4}X^{8})^2d\alpha d\beta\\ \ll
&L^{8}Q^{-8}X^{17}\int_{\mathfrak{m}(Q)}\int_{\mathfrak{m}(Q)}\sum_{|x|\le
X}\min\{X,\ \|\frac{\alpha-\beta}{b_5}x\|^{-1}\}d\alpha
d\beta.\end{align*} Then we conclude from Lemma \ref{lemma61} that
\begin{align}\label{cru}\sum_{|z_1|\le X}\sum_{|z_2|\le X}\sum_{\substack{|k|\le
c'X}} \Big|\int_{\mathfrak{m}(Q)}e\big(\frac{\alpha}{b_5}
(z_1-z_2)k\big)\sum_{\substack{|\mathbf{u}|\le
cX}}\Pi(\alpha,\mathbf{u},z_1,z_2)d\alpha\Big|^2 \ll
L^{9}Q^{-\frac{9}{2}}X^{15}.\end{align} By (\ref{meanxi2}) and
(\ref{cru}),
\begin{align}\int_{\mathfrak{m}(Q)}\Xi_{\gamma}(\alpha)d\alpha \ll L^{17/2}Q^{-9/4}X^{9}.\label{meanxi}\end{align}
By substituting (\ref{meanxi}) into (\ref{intJ}), we obtain
\begin{align}\label{nonzero}\int_{\mathfrak{m}(Q)}J_{\gamma}(\alpha)
d\alpha \ll L^{25/4}Q^{-17/8}X^{8}\end{align} provided that
$b_5\not=0$.

We complete the proof in view of the argument around (\ref{zero})
and (\ref{nonzero}).
\end{proof}

\begin{lemma}\label{lemma64}One
has\begin{align*}\int_{\mathfrak{m}(Q)}\big|S(\alpha)\big| d\alpha
\ll L^{n+1}Q^{-1/16}X^{n-2}.\end{align*}\end{lemma}
\begin{proof}By Cauchy's inequality,
\begin{align}\label{eq618}\int_{\mathfrak{m}(Q)}\big|S(\alpha)\big| d\alpha \le
 & \Big(\int_{\mathfrak{m}(Q)} d\alpha
\Big)^{1/2}\Big(\int_{\mathfrak{m}(Q)}\big|S(\alpha)\big|^2
d\alpha \Big)^{1/2}\\ \ll &
QX^{-1}\Big(\int_{\mathfrak{m}(Q)}\big|S(\alpha)\big|^2 d\alpha
\Big)^{1/2}.\end{align} It follows from Lemma
\ref{lemma62}-\ref{lemma63}
that\begin{align}\label{eq619}\int_{\mathfrak{m}(Q)}\big|S(\alpha)\big|
d\alpha \ll L^{2n+1}
Q^{-17/8}X^{2n-2}\int_{0}^1\Phi(\gamma)d\gamma \ll L^{2n+2}
Q^{-17/8}X^{2n-2}.\end{align}We complete the proof by putting
(\ref{eq618}) into (\ref{eq619}).
\end{proof}

We finish Section 6 by pointing out that Proposition
\ref{proposition61} follows from Lemma \ref{lemma64} by the dyadic
argument.

 \vskip3mm

\section{Proof of Theorem \ref{theorem1}}

By orthogonality, we have
\begin{align*}N_{f,t}(X)=\int_{X^{-1}}^{1+X^{-1}}S(\alpha)e(-t\alpha)d\alpha.\end{align*}
Recalling the definitions of $\mathfrak{M}$ and $\mathfrak{m}$ in
(\ref{defM}) and (\ref{defm}), we have
\begin{align}\label{NX}N_{f,t}(X)=\int_{\mathfrak{M}}S(\alpha)e(-t\alpha)d\alpha
+\int_{\mathfrak{m}}S(\alpha)e(-t\alpha)d\alpha.\end{align} In
light of Lemma \ref{lemma34}, to establish the asymptotic formula
(\ref{asymptotic}), it suffices to
prove\begin{align}\label{minor}\int_{\mathfrak{m}}\big|S(\alpha)\big|
d\alpha \ll X^{n-2}L^{-K/20}.\end{align} In view of Proposition
\ref{proposition61} and the work of Liu \cite{Liu} (see also
Remark of Lemma \ref{lemma44}), the estimate (\ref{minor}) holds
if there exists an invertible matrix
\begin{align*}B=\begin{pmatrix}a_{i_1,j_1} & \cdots & a_{i_5,j_5}
\\ \vdots & \cdots & \vdots
\\ a_{i_5,j_1} & \cdots & a_{i_5,j_5}
\end{pmatrix}\end{align*} with $$|\{i_1,\ldots,i_5\}\cap\{j_1,\cdots,j_5\}|\le 1.$$ Next
we assume $\rank(B)\le 4$ for all $B=(a_{i_k,j_l})_{1\le k,l\le
5}$ satisfying $|\{i_1,\ldots,i_5\}\cap\{j_1,\cdots,j_5\}|\le 1$.
This yields $\rankoff(A)\le 4$. By Proposition
\ref{proposition51}, we can establish (\ref{minor}) again if
$\rankoff(A)\le 3$. It remains to consider the case
 $\rank_{\textrm{off}}(A)\geqslant 4$. Without loss of
generality, we assume that $\rank(C)=4$,
where\begin{align*}C=\begin{pmatrix}a_{1,5} & a_{1,6} & a_{1,7} &
a_{1,8}
\\ a_{2,5} & a_{2,6} & a_{2,7} &
a_{2,8}
\\ a_{3,5} & a_{3,6} & a_{3,7} &
a_{3,8}
\\ a_{4,5} & a_{4,6} & a_{4,7} &
a_{4,8}\end{pmatrix}.\end{align*}Let $\gamma_j=(a_{j,5}, \ldots,
a_{j,n})^{T}\in \Z^{n-4}$ for $1\le j\le n$. Then $\gamma_1$,
$\gamma_2$, $\gamma_3$ and $\gamma_4$ are linear independent due
to $\rank(C)=4$. For $5\le k\le n$, we considering
\begin{align*}B=\begin{pmatrix}a_{1,5} &
\cdots & a_{1,n}
\\ \vdots & \cdots & \vdots
\\ a_{4,5} & \cdots & a_{4,n}
\\ a_{k,5} & \cdots & a_{k,n}\end{pmatrix}\in
M_{5,n-4}(\Z).\end{align*}According to our assumption, one has
$\rank(B)\le 4$. Then we conclude from above that $\gamma_k$ can
be linear represented by $\gamma_1$, $\gamma_2$, $\gamma_3$ and
$\gamma_4$. Therefore, one has $\rank(H)=4$, where
\begin{align*}H=\begin{pmatrix}a_{1,5} & \cdots & a_{1,n}
\\ \vdots & \cdots & \vdots
\\ a_{n,5} & \cdots & a_{n,n}\end{pmatrix}\in M_{n,n-4}(\Z).\end{align*}
We obtain $\rank(A)\le \rank(H)+4\le 8$. This is contradictory to
the condition that $\rank(A)\geqslant 9$. Therefore, we complete
the proof of Theorem \ref{theorem1}.


\vskip4mm
\end{document}